\documentclass[10pt,a4paper]{article}
\newcommand{\vect}[1]{\boldsymbol{#1}}
\newcommand{\vectt}[1]{\boldsymbol{\mathbf{#1}}}

\newcommand{\dx}{\mathrm{d}x}

\newcommand{\supp}{\mathrm{supp}}

\newcommand{\eP}{\tilde{P}}

\newcommand{\E}{\mathrm{e}}

\def\Xint#1{\mathchoice
{\XXint\displaystyle\textstyle{#1}}%
{\XXint\textstyle\scriptstyle{#1}}%
{\XXint\scriptstyle\scriptscriptstyle{#1}}%
{\XXint\scriptscriptstyle\scriptscriptstyle{#1}}%
\!\int}
\def\XXint#1#2#3{{\setbox0=\hbox{$#1{#2#3}{\int}$}
\vcenter{\hbox{$#2#3$}}\kern-.5\wd0}}

\def\dashint{\Xint-}

\usepackage{color}
\definecolor{deepblue}{rgb}{0,0,0.5}
\definecolor{deepred}{rgb}{0.6,0,0}
\definecolor{deepgreen}{rgb}{0,0.5,0}

\usepackage{amsmath,amssymb}
\usepackage{bm}
\usepackage{amsthm}
\usepackage{anyfontsize}
\usepackage{url}
\usepackage{graphicx}
\usepackage{array}
\usepackage{varwidth}
\usepackage{geometry}
\usepackage{ esint }
\definecolor{ao(english)}{rgb}{0.0, 0.5, 0.0}
\usepackage{listings}
\usepackage{fancyhdr}
\usepackage{enumitem}
\usepackage[
    backend=bibtex,
   bibstyle=ieee,
    citestyle=numeric-comp,
    natbib=true,
    sortlocale=en_GB,
    sorting = nyt,
    url=true,
    doi=true,
    arxiv=true,
    eprint=true
]{biblatex}
\DeclareFieldFormat{eprint:arXiv}{%
  arXiv preprint: \href{https://arxiv.org/abs/#1}{#1 \texttt{[math.NA]}}%
}
\usepackage{caption}	
\usepackage{mathtools}
\usepackage{subfig}
\usepackage{appendix}
\usepackage{tikz-cd} 
\usepackage{hyperref}
\usepackage{xcolor}
\hypersetup{
    colorlinks,
    linkcolor={blue!50!black},
    citecolor={blue!50!black},
    urlcolor={blue!80!black}
}

\RequirePackage{algorithm}[0.1]

\RequirePackage[capitalize,nameinlink]{cleveref}[0.19]
\crefformat{equation}{\textup{#2(#1)#3}}
\crefrangeformat{equation}{\textup{#3(#1)#4--#5(#2)#6}}
\crefmultiformat{equation}{\textup{#2(#1)#3}}{ and \textup{#2(#1)#3}}
{, \textup{#2(#1)#3}}{, and \textup{#2(#1)#3}}
\crefrangemultiformat{equation}{\textup{#3(#1)#4--#5(#2)#6}}%
{ and \textup{#3(#1)#4--#5(#2)#6}}{, \textup{#3(#1)#4--#5(#2)#6}}{, and \textup{#3(#1)#4--#5(#2)#6}}

\Crefformat{equation}{#2Equation~\textup{(#1)}#3}
\Crefrangeformat{equation}{Equations~\textup{#3(#1)#4--#5(#2)#6}}
\Crefmultiformat{equation}{Equations~\textup{#2(#1)#3}}{ and \textup{#2(#1)#3}}
{, \textup{#2(#1)#3}}{, and \textup{#2(#1)#3}}
\Crefrangemultiformat{equation}{Equations~\textup{#3(#1)#4--#5(#2)#6}}%
{ and \textup{#3(#1)#4--#5(#2)#6}}{, \textup{#3(#1)#4--#5(#2)#6}}{, and \textup{#3(#1)#4--#5(#2)#6}}

\newmuskip\pFqmuskip

\newcommand*\pFq[6][8]{%
  \begingroup 
  \pFqmuskip=#1mu\relax
  \mathchardef\normalcomma=\mathcode`,
  \mathcode`\,=\string"8000
  \begingroup\lccode`\~=`\,
  \lowercase{\endgroup\let~}\pFqcomma
  {}_{#2}F_{#3}{\left(\genfrac..{0pt}{}{#4}{#5};#6\right)}%
  \endgroup
}
\newcommand{\pFqcomma}{{\normalcomma}\mskip\pFqmuskip}

\usepackage{algpseudocode}
\algnewcommand{\Input}[1]{
  \State \textbf{Input:}
  \Statex \hspace*{\algorithmicindent}\parbox[t]{.8\linewidth}{\raggedright #1}
}
\algnewcommand{\Output}[1]{
  \State \textbf{Output:}
  \Statex \hspace*{\algorithmicindent}\parbox[t]{.8\linewidth}{\raggedright #1}
}
\algnewcommand{\Indent}[2]{
  \State {#1}
  \vspace{-2mm}
  \Statex \hspace*{\algorithmicindent}\parbox[t]{.9\linewidth}{\raggedright #2}
}
\renewcommand{\Comment}[2][.5\linewidth]{%
  \leavevmode\hfill\makebox[#1][l]{$\triangleright$~#2}}

\usepackage{todonotes}

\newtheorem{proposition}{Proposition}[section]
\newtheorem{theorem}{Theorem}[section]
\newtheorem{lemma}{Lemma}[section]
\newtheorem{definition}{Definition}[section]
\newtheorem{corollary}{Corollary}[section]
\newtheorem{remark}{Remark}[section]
\numberwithin{equation}{section}

\usepackage{todonotes}

\title{A frame approach for equations  involving \\ the fractional Laplacian}

\author{
Ioannis P.~A.~Papadopoulos\thanks{Weierstrass Institute for Applied Analysis and Stochastics, DE,  \tt{papadopoulos@wias-berlin.de};} \and 
Timon S.~Gutleb \thanks{School of Computer Science, University of Leeds, UK, {\tt t.s.gutleb@leeds.ac.uk;}} \and
Jos\'e A.~Carrillo \thanks{Mathematical Institute, University of Oxford, UK, {\tt carrillo@maths.ox.ac.uk};} \and
Sheehan Olver \thanks{Department of Mathematics, Imperial College London, UK, {\tt{s.olver@imperial.ac.uk}}.}}

\begin{document}

\maketitle
\date
\thispagestyle{empty}
\pagestyle{fancy}
\lhead{}

\begin{abstract}
Exceptionally elegant formulae  exist for the fractional Laplacian operator applied to weighted classical orthogonal polynomials. We utilize these results to construct  a solver,  based on frame properties, for equations involving the fractional Laplacian of any power, $s \in  (0,1)$, on an unbounded domain in one or two dimensions. The numerical method  represents solutions in  an expansion of weighted classical orthogonal polynomials as well as their unweighted counterparts with a specific extension to $\mathbb{R}^d$, $d \in \{1,2\}$.    We examine the frame properties of this family of functions for the solution expansion and, under standard frame conditions, derive an a priori estimate for the stationary equation. Moreover, we prove  one achieves the expected order of convergence when considering an implicit Euler discretization in time for the fractional heat equation. We apply our solver to numerous examples including the fractional heat equation (utilizing up to a $6^\text{th}$-order Runge--Kutta time discretization), a fractional heat equation with a time-dependent exponent $s(t)$, and a two-dimensional problem, observing spectral convergence in the spatial dimension for sufficiently smooth data.
\end{abstract}

\section{Introduction}
\label{sec:introduction}

In this work, we develop a spectral method to solve equations of the form
\begin{align}
\mathcal{L}[u] \coloneqq  \lambda_0 \mathcal{I}  + \sum_{j =1}^{ J } \lambda_{ j } (-\Delta)^{s_{ j }}[u] = f,
\label{eq:fpde}
\end{align}
where $\lambda_{ j } \in \mathbb{R}$ are known constants and $s_{ j } \in  (0,1)$.  $(-\Delta)^{s}$ denotes the fractional Laplacian, as defined in the next section, and $\mathcal{I}$ is defined to be the identity operator. The domain is the whole real space $\mathbb{R}^d$, $d \in \{1,2\}$. For uniqueness, we seek a solution that decays as $|x| \to \infty$.

Many equations may reduce to \cref{eq:fpde} after nondimensionalization, linearization, or partial (time) discretization \cite{Antil2017, Du2019, Hatano1998, Benson2000, Carmichael2015}. Solvers for \cref{eq:fpde} are an intermediate step for more complicated problems such as  Newton linearizations of nonlocal Burgers-type equations \cite{Baker1996, Cordoba2006}, power law absorption \cite{Treeby2010}, quasi-geostrophic equations \cite{Chae2005}, fractional Keller-Segel type models \cite{E06,LR09,BC10,LS19}, or fractional porous medium flow \cite{Caffarelli2016}. Finite element methods are popular for approximating the solutions of these equations \cite{Bonito2018, LPG2020}. However, since the operators are nonlocal and the solutions tend to be heavy-tailed (for instance decay at the rate of a Cauchy distribution when $s=1/2$) \cite[Sec.~2]{Vazquez2018}, then any truncation of $\mathbb{R}^d$ may result in artificial numerical artefacts \cite[Sec.~9]{D2020}. Spectral methods that construct bases with global support can be found in the literature \cite{Chen2018, Tang2018, Tang2020, Mao2017, Sheng2020, Li2021, Cayama2020, Papadopoulos2022e}. These alleviate the need to truncate $\mathbb{R}^d$, $d \in \{1,2,3\}$, and some observe spectral convergence for sufficiently smooth data \cite{D2020}. 

Our goal is to construct a flexible spectral method for \cref{eq:fpde} posed on $\mathbb{R}^d$, $d \in \{1,2\}$, via a frame approach \cite{Adcock2019, Adcock2020}.  Our algorithm is summarized in \cref{alg:frame}.

\begin{algorithm}[ht]
\caption{Frame solver for \cref{eq:fpde}.}
\label{alg:frame}
\begin{algorithmic}[1]

\Input{
$\mathcal{L}, f$. \Comment{Operator and right-hand side in \cref{eq:fpde}.}\\
$\Phi_N=\{\phi_1,\dots,\phi_N\}$. \Comment{Set of functions for approximating $u$.}\\
$\{\ell_1, \dots, \ell_M\}$. \Comment{Set of bounded linear functionals.}\\
$\epsilon$. \Comment{SVD tolerance.}
\vspace{2mm}
}
\Output{
${\mathbf u} \in \mathbb{R}^N$. \Comment{$u \approx \sum_{i=1}^N {\mathbf u}_i \phi_i$.}
}
\State{Assemble the matrix $X_{ij} = \ell_j(\mathcal{L}\phi_i)$, $i=1,\dots,N$, $j=1,\dots,M$.}
\State{Assemble the vector ${\mathbf y}_j = \ell_j(f)$, $j=1,\dots,M$.}
\State{Via an $\epsilon$-truncated SVD projection, compute ${\mathbf f} \approx \mathrm{argmin}_{{\mathbf v}\in \mathbb{C}^N} \|X {\mathbf v} - {\mathbf y}\|_{\ell^2}$.}
\State{${\mathbf u} \gets {\mathbf f}$.}
\end{algorithmic}
\end{algorithm}
The performance of \cref{alg:frame} for solving \cref{eq:fpde} is heavily dependent on the choice of bounded linear functionals $\ell_j$ and the family of functions $\Phi_N$ with which we approximate $u$.  In the numerical examples in \cref{sec:examples}, we focus on right-hand sides that are continuous. Hence, we choose a number of collocation points $x_j$, $j \in \{1,\dots, M\}$ and assign $\ell_j (f) = f(x_j)$.

\begin{remark}[General right-hand sides]
For a more general right-hand side $f \in H^*$ that cannot be evaluated pointwise, one may instead consider an orthonormal family of functions on $\mathbb{R}$, $\{\zeta_j \} \subset H$, and choose for $j=1,2,\dots$,
\begin{align}
{\mathbf y}_j = \ell_j(f) = \langle f, \zeta_{j} \rangle_{H^*, H}.
\label{eq:f1}
\end{align}
For instance $\zeta_j(x) = H_{n-1}(x)$ where $\{H_n\}$ are the orthonormalized Hermite polynomials \cite[Sec.~18.3]{OlverNIST}. Similarly one must then construct the matrix (line 3 in \cref{alg:frame})
\begin{align}
X_{ij} = \ell_j(\mathcal{L} \phi_i) = \langle \mathcal{L}\phi_i, \zeta_{j}\rangle_{H^*, H}.
\label{eq:f2}
\end{align}
Theoretically \labelcref{eq:f1} and \labelcref{eq:f2} will fit the criteria of \cref{th:apriori} \cite[Sec.~5.1]{Adcock2020}. Practically, however, the entries may be difficult to evaluate. 
\end{remark}

A crucial ingredient  for the construction of $\Phi_N$  are recent results concerning the application of $(-\Delta)^s$ to weighted and extended Jacobi, $P^{(s,s)}_n(x)$, and Zernike polynomials, $Z^{(b)}_{n,m,j}(x,y)$ \cite{Gutleb2023}. By fixing $\Phi_N$, we induce the family of functions for the right-hand side expansion $\mathcal{L}\Phi_N = (\mathcal{L}\phi_1, \dots, \mathcal{L}\phi_N)$. Thus the problem reduces from solving a fractional differential equation to expanding the right-hand side $f$ in the induced family of functions. By noticing this space has the qualities of a frame, this motivates obtaining the expansion by solving a least-squares problem via an $\epsilon$-truncated SVD projection  as done in line 5 of \cref{alg:frame}.  More precisely, in \cref{prop:frame} we show that the family of functions consisting of extended and weighted Jacobi polynomials is a frame on the Hilbert space $H^s_w(\mathbb{R})$,  $s \in (0,1)$,  defined in \cref{def:Hsw}. Then, by utilizing \cref{lem:range}, in \cref{cor:rhs-frame} we deduce that the operator $(\mathcal{I}+(-\Delta)^s)$,  $s \in (0,1)$,  induces a family of functions for the expansion of the right-hand side that is a frame on the dual space of $H^s_w(\mathbb{R})$.  This leads to the a priori estimate in \cref{th:apriori}.  Although we do not show that the families of functions are frames on $L^2(\mathbb{R})$, we observe that if the correct family of functions is chosen, and a sufficient number of collocation points are sampled, then the coefficients of the right-hand side expansion are well-behaved. Moreover, the coefficients of the expansion for the solution $u$ are immediately deduced and the convergence is spectral for smooth right-hand sides.

\begin{remark}[Scope of the results]
Our frame results only hold for the one-dimensional case for a specific family of functions. However, experimentally, our solver is effective for more general families of functions as well as higher dimensions. Hence, although our analysis is restricted to the case specified in the statement of  \cref{prop:frame}, we apply our solver to a multitude of different examples in \cref{sec:examples}.
\end{remark}

If the number of bounded linear functionals and the number of functions in the frame are the same then, mathematically, the solver is very similar to a collocation method e.g.~\cite{Tang2018, Tang2020}. The main differences are the flexible choices for the solution family of functions, the smaller overhead for the setup of the method, and the construction as a frame approach. A critique is that the truncated SVD of an $m \times n$ matrix, $m > n$, scales at $\mathcal{O}(mn^2)$. For the purposes of this work, the factorization time is negligible and does not influence our results. However, the focus of future work will be to achieve close to $\mathcal{O}(n\log^2 n)$ operations via the $AZ$-algorithm \cite{Coppe2020} and randomized linear algebra solvers for least-squares \cite{Halko2011, Meier2022, Meier2023}.

In \cref{sec:setup} we rigorously state our problem and provide the relevant mathematical background. Then:
\begin{enumerate}
\itemsep=0pt
\item In \cref{sec:spectral-method} we explain why an expansion of the right-hand side $f$, as conducted in line 5 of \cref{alg:frame}, provides an approximation for the solution.
\item We introduce the weighted Jacobi and Zernike polynomials in \cref{sec:ops} and define the functions that are found in the induced family of functions for the right-hand side expansion.
\item In \cref{sec:frames} we prove a number of results that contextualizes our method as a frame approach. These results motivate the use of an $\epsilon$-truncated SVD projection to expand $f$. We end with an a priori estimate in \cref{th:apriori}.
\item With these tools in \cref{sec:time-dependent} we provide the algorithm for the time discretization via Runge--Kutta methods, which is subtly different to a standard approach. We also prove that, in the case of a backward Euler discretization, we obtain the expected convergence rates in time and space.
\item We motivate our spectral method in several examples in \cref{sec:examples} including equations with multiple fractional Laplacians with different exponents, a variable exponent, and a two-dimensional example.
\end{enumerate}

\section{Mathematical setup}
\label{sec:setup}

Let $W^{s,p}(\mathbb{R}^d)$, $d \in \{1,2\}$, $s>0$, denote the (possibly fractional) Sobolev space \cite{Adams2003,Di2012} and $H^{s}(\mathbb{R}^d) \coloneqq W^{s,2}(\mathbb{R}^d)$. We denote the Lebesgue space by  $L^p(\mathbb{R}^d)$, $p \in [1,\infty]$.  We seek a solution $u \in H^{s}(\mathbb{R}^d)$ for \cref{eq:fpde}. Let $H^{-s}(\mathbb{R}^d) \coloneqq H^s(\mathbb{R}^d)^*$ denote the dual space of $H^s(\mathbb{R}^d)$. Then we require $f \in H^{-s}(\mathbb{R}^d)$. Given a Banach space $V$ equipped with the norm $\| \cdot \|_V$, we define the dual norm of the dual space $V^*$ as follows, for any $u^* \in V^*$,
\begin{align}
\| u^* \|_{V^*} \coloneqq \sup_{u \in V, \; u \neq 0} \frac{|\langle u^*, u \rangle_{V^*, V}|}{\| u \|_V}.
\end{align}
It follows that for any $u^* \in V^*$ and $v \in V$, $\langle u^*, v \rangle_{V^*, V} \leq \| u^*\|_{V^*} \| v \|_V$. We denote the space of bounded linear operators between the Banach spaces $V$ and $W$ by $\mathcal{B}(V, W)$ and equip this space with the operator norm, for any $\mathcal{A} \in \mathcal{B}(V,W)$,
\begin{align}
\| \mathcal{A} \|_{\mathcal{B}(V,W)} \coloneqq \sup_{u \in V, \; u \neq 0} \frac{\| \mathcal{A}u \|_W}{\| u \|_V}.
\end{align}
Similarly, we have that for any $\mathcal{A} \in  \mathcal{B}(V,W)$ and $v \in V$, $\| \mathcal{A}u \|_W \leq \| \mathcal{A} \|_{\mathcal{B}(V,W)} \|v\|_V$.

In this work, $\mathcal{I}$ denotes the identity operator and $(-\Delta)^{s}$ denotes the fractional Laplacian. The fractional Laplacian has multiple definitions, many of them equivalent on unbounded domains when considering sufficiently regular functions \cite{Kwasnicki2017} (and not equivalent in other contexts).  Let $\mathcal{S}(\mathbb{R}^d)$ denote the space of Schwartz functions on $\mathbb{R}^d$, $d \in \{1,2\}$. For any $s \in (0,1)$, we define $(-\Delta)^{s} : \mathcal{S}(\mathbb{R}^d) \to L^2(\mathbb{R}^d)$ as \cite[Sec.~3]{Di2012}:
\begin{align}
(-\Delta)^{s} u(\vect{x}) &\coloneqq c_{d,s} \,\,  \dashint_{\mathbb{R}^d} \frac{u(\vect{x}) - u(\vect{y})}{|\vect{x}-\vect{y}|^{d+2s}} \, \mathrm{d}\vect{y}, \; \text{for a.e.} \; \vect{x} \in \mathbb{R}^d, \;\;
c_{d,s} \coloneqq \frac{4^{s} \Gamma(d/2+s)}{\pi^{d/2} |\Gamma(-s)|}. \label{def:fracLap}
\end{align} 
Here $\dashint_{\mathbb{R}^d} \cdot$ denotes the Cauchy principal value integral \cite[Ch.~2.4]{King2009a} and $\Gamma(\cdot)$ denotes the Gamma function. The normalization factor $c_{d,s}$ ensures that as $s\to 0_+$ and $s \to 1_-$, then $(-\Delta)^s u$ converges to $u$ and $-\Delta u$, for any $u\in C^\infty_0(\mathbb{R}^d)$, respectively \cite[Prop.~4.4]{Di2012}. For any $u \in H^{s}(\mathbb{R}^d)$, it may be shown that $(-\Delta)^{s/2}u \in L^2(\mathbb{R}^d)$ \cite[Prop.~3.6]{Di2012}. The action of the fractional Laplacian  on any $u \in H^s(\mathbb{R}^d)$  may be recast in the weak form, $(-\Delta)^{s}: H^{s}(\mathbb{R}^d) \to H^{-s}(\mathbb{R}^d)$:
\begin{align}
\langle (-\Delta)^{s} u, v \rangle_{H^{-s}(\mathbb{R}^d),H^{s}(\mathbb{R}^d)} &= \langle (-\Delta)^{s/2} u,(-\Delta)^{s/2}  v \rangle_{L^2(\mathbb{R}^d)} \;\; \text{for all} \; v \in H^{s}(\mathbb{R}^d). \label{eq:weak}
\end{align}
 Fix $J=1$ in \cref{eq:fpde} and let $\lambda_0 = \lambda$, $\lambda_1 = 1$, and $s_1 = s$ for some $\lambda > 0$ and $s \in (0,1)$.  Consider the weak form of \cref{eq:fpde}, for a given $f \in H^{-s}(\mathbb{R}^d)$ and for all $v \in H^{s}(\mathbb{R}^d)$,
\begin{align}
 \langle (-\Delta)^{s/2} u,(-\Delta)^{s/2}  v \rangle_{L^2(\mathbb{R}^d)} + \lambda \langle u, v \rangle_{L^2(\mathbb{R}^d)} = \langle f, v \rangle _{H^{-s}(\mathbb{R}^d),H^{s}(\mathbb{R}^d)}. \label{eq:weak-form}
\end{align}
The existence of a unique solution $u_* \in H^{s}(\mathbb{R}^d)$ of \cref{eq:weak-form} may be shown via the Lax--Milgram theorem, e.g.~\cite[Sec.~2.2]{Mao2017} and \cite[Thm.~2.2]{Papadopoulos2022e}. By definition of $H^{s}(\mathbb{R}^d)$, we have that $u_* \in L^2(\mathbb{R}^d)$. Moreover, if $f \in L^2(\mathbb{R}^d)$ (and thus $(-\Delta)^{s} u_* \in L^2(\mathbb{R}^d)$) then $u_*$ also solves \cite[Th.~1.1]{Kwasnicki2017}
\begin{align}
(\lambda \mathcal{I} + (-\Delta)^{s}) u_* = f \; \text{a.e.~in} \; \mathbb{R}^d. \label{eq:strong}
\end{align}

\section{The spectral method}
\label{sec:spectral-method}
In this section  we motivate \cref{alg:frame} as an approach for solving \cref{eq:fpde}. In essence we are diagonalizing the operator $\mathcal{L}$.

\begin{definition}[Quasimatrix]
Throughout this work, we use boldface capital letters to denote a quasimatrix associated with a set of functions. A quasimatrix is a matrix whose ``columns'' are functions defined on $\mathbb{R}^d$ \cite[Lec.~5]{Stewart1998}. For instance if $\Phi = \{\phi_i\}_{i=1}^\infty$ then
\begin{align}
{\bm \Phi}(x) &\coloneqq 
\left( 
\phi_1(x) \ \  \phi_2(x) \ \ \cdots
\right).
\end{align}
\end{definition}

\begin{proposition}
\label{prop:apriori2}
Consider the equation \cref{eq:fpde}, let $H$ be a Hilbert space, and fix a right-hand side $f \in H^*$. Suppose that $\mathcal{L}^{-1} \in \mathcal{B}(H^*, H)$. Consider a  finite set of functions  $\Phi_N = \{\phi_1, \dots, \phi_N\}$ and define $\Psi_N = \mathcal{L} \Phi_N = \{\mathcal{L}\phi_1, \dots, \mathcal{L}\phi_N\}$. Consider any expansion coefficient ${\mathbf f}$ for $f$ in the family of functions $\Psi_N$. Then
\begin{align}
\| u - {\bm \Phi}_N {\mathbf f} \|_H \leq \| \mathcal{L}^{-1}\|_{\mathcal{B}(H^*, H)} \| f - {\bm \Psi}_N {\mathbf f} \|_{H^*}.
\end{align}
\end{proposition}
\begin{proof}
\begin{align}
\begin{split}
\| u - {\bm \Phi}_N {\mathbf f} \|_H 
&= \| \mathcal{L}^{-1} \mathcal{L}u - \mathcal{L}^{-1} \mathcal{L} {\bm \Phi}_N {\mathbf f} \|_H \\
&= \| \mathcal{L}^{-1} f- \mathcal{L}^{-1} {\bm \Psi}_N {\mathbf f} \|_H
\leq \| \mathcal{L}^{-1}\|_{\mathcal{B}(H^*, H)} \| f - {\bm \Psi}_N {\mathbf f} \|_{H^*}.
\end{split}
\end{align}
\end{proof}
\cref{alg:frame} is approximately finding an expansion for $f$  via an $\epsilon$-truncated SVD projection and thus \cref{prop:apriori2} implies that by setting ${\mathbf u} = {\mathbf f}$, then $\| u - {\bm \Phi}_N {\mathbf u} \|_H  \leq \| \mathcal{L}^{-1}\|_{\mathcal{B}(H^*, H)} \| f - {\bm \Psi}_N {\mathbf f} \|_{H^*}$. 

In order to expand $f$, we must know the exact expressions for the functions in $\Psi_N = \mathcal{L}\Phi_N$. In the next section, we define functions such that, in the context of \cref{eq:fpde}, these expressions are known.

\section[Orthogonal polynomials]{Extended and weighted orthogonal polynomials}
\label{sec:ops}

\subsection{Jacobi polynomials}
\label{sec:jacobi}

The Jacobi polynomials $\{ P^{(a,b)}_n(x) \}_{n\in \mathbb{N}}$ are a family of complete univariate bases of classical orthogonal polynomials on the interval $(-1,1)$ with weight parameters $a, b \in \mathbb{R}$ such that $a,b >-1$. They are orthogonal with respect to the weight $(1-x)^a (1+x)^b$. Let ${_2}F_1(\cdot,\cdot;\cdot;\cdot)$ and $(\cdot)_n$ denote the hypergeometric function \cite[Sec.~15.2]{OlverNIST} and Pochhammer symbol \cite[Sec.~5.2(iii)]{OlverNIST}, respectively. We pick the standard normalization such that \cite[Eq.~18.5.7]{OlverNIST}
\begin{align}
P^{(a,b)}_n(x) \coloneqq \frac{(a+1)_n}{n!} {}_{2}F_1\left(-n, 1+a+b+n;a+1;\frac{1}{2}(1-x)\right).
\end{align}
Let $(\cdot)_+$ denote the max function $(z)_+ \coloneqq \mathrm{max}(0,z)$ \cite[1.16.19]{OlverNIST}. Then we define the weighted Jacobi polynomial $Q^{(a,b)}_n(x)$ as 
\begin{align}
Q^{(a,b)}_n(x) \coloneqq (1-x)_+^a (1+x)_+^b P^{(a,b)}_n(x).
\end{align}
For any $x \in \mathbb{R} \backslash [-1,1]$, we fix $P^{(a,b)}_n(x) = Q^{(a,b)}_n(x) = 0$ for all $n \in \mathbb{N}_0$. 

The extended Jacobi functions $\eP^{(a,s)}_n(x)$, $s \in (-1, 0) \cup (0,1)$, $a>-1$, are defined as
\begin{align}
\eP^{(a,s)}_n(x) \coloneqq (-\Delta)^s Q^{(a,a)}_n(x). 
\label{def:extended-Jacobi}
\end{align}
Let $\Gamma(\cdot)$ denote the Gamma function. The following theorem provides explicit expressions for $\eP^{(a,s)}_n(x)$.
\begin{theorem}[Row 1, Tab.~3 in \cite{Gutleb2023}]
Consider $s \in (-1/2,0)\cup (0,1)$ and $a > -1$. Then, for any  $n \in \mathbb{N}_0$  and $x \in \mathbb{R}$, $|x| \neq 1$,
\begin{align}
\begin{split}
&\eP^{(a,s)}_n(x)
=
4^{s}\frac{\Gamma (a+n+1)}{n!} x^{n-2 \left\lfloor \frac{n}{2}\right\rfloor }\\
&\times
\resizebox{0.85\textwidth}{!}{$
\begin{cases}
\footnotesize
\frac{\pi \, {}_2F_1\left(-a+s-\left\lfloor \frac{n}{2}\right\rfloor, n+s-\left\lfloor \frac{n}{2}\right\rfloor +\frac{1}{2};n-2 \left\lfloor \frac{n}{2}\right\rfloor +\frac{1}{2};x^2\right)}{\sin \left(\pi  \left(2 \left\lfloor \frac{n}{2}\right\rfloor - n- s+\frac{1}{2}\right)\right)\Gamma \left(n-2 \left\lfloor \frac{n}{2}\right\rfloor +\frac{1}{2}\right) \Gamma \left(-n-s+\left\lfloor \frac{n}{2}\right\rfloor +\frac{1}{2}\right) \Gamma \left(a-s+\left\lfloor \frac{n}{2}\right\rfloor +1\right)} & |x|<1,  \\
-\frac{2^{-n-2 s} \sin (\pi  s) \Gamma (n+2 s+1) | x| ^{-2 \left\lfloor \frac{n-1}{2}\right\rfloor -2 s-3} \, _2F_1\left(s+\left\lfloor \frac{n}{2}\right\rfloor +1,\frac{2 \left\lfloor \frac{n-1}{2}\right\rfloor +3}{2} +s;\frac{2 n+3}{2}+a;\frac{1}{x^2}\right)}{\sqrt{\pi } \Gamma \left(\frac{2 n+3}{2} +a\right)} & |x|>1. \end{cases}$}
\end{split}
\label{eq:extendedJacobidef}
\end{align}
\end{theorem}

When $a=s$, the expression \cref{eq:extendedJacobidef} simplifies. We find that, for $x \in (-1,1)$, $\eP^{(s,s)}_n(x)$ is equal, up to a constant, to the corresponding Jacobi polynomial\footnote{One could instead re-normalize $\eP^{(s,s)}_n$ such that $\eP^{(s,s)}_n(x) = P^{(s,s)}_n(x)$ for $x \in (-1,1)$. However, this has the unfortunate consequence of making the notation more obfuscated in the later results.} $P^{(s,s)}(x)$. 
\begin{theorem}[Row 1$^*$, Tab.~5 in \cite{Gutleb2023}]
\label{thm:fractionalwholespaceultra}
Consider $s \in (-1/2,0) \cup (0,1)$. Then, for any  $n \in \mathbb{N}_0$  and $x \in \mathbb{R}$, $|x| \neq 1$,
\begin{align}
\eP^{(s,s)}_{n}(x) = 
\resizebox{0.8\textwidth}{!}{$
\begin{cases}  \frac{4^s \Gamma \left(s+\left\lfloor \frac{n}{2}\right\rfloor +1\right) \Gamma \left(n+s-\left\lfloor \frac{n}{2}\right\rfloor +\frac{1}{2}\right)}{\left\lfloor \frac{n}{2}\right\rfloor ! \Gamma \left(n-\left\lfloor \frac{n}{2}\right\rfloor +\frac{1}{2}\right)} P_n^{(s,s)}(x) &|x|<1,  \\
-\frac{ \sin (\pi  s) x^{n-2 \left\lfloor \frac{n}{2}\right\rfloor } \Gamma (n+s+1) \Gamma (n+2 s+1) | x| ^{-2 \left\lfloor \frac{n-1}{2}\right\rfloor -2 s-3} \, _2F_1\left(s+\left\lfloor \frac{n-1}{2}\right\rfloor +\frac{3}{2},s+\left\lfloor \frac{n}{2}\right\rfloor +1;n+s+\frac{3}{2};\frac{1}{x^2}\right)}{2^n \sqrt{\pi } n! \Gamma \left(n+s+\frac{3}{2}\right)} & |x|>1.  \end{cases}$}
\end{align}
\end{theorem}
\begin{remark}
We see that for $x \in (-1,1)$, \cref{thm:fractionalwholespaceultra} implies that $\eP_n^{(s,s)}$ are in fact scaled Jacobi polynomials for $x \in (-1,1)$. This is a key observation in later results.
\end{remark}

\begin{remark}
Thanks to their definition, we expect a solution family of functions consisting of extended Jacobi functions to accurately capture the algebraic decay of the solution tails. It can be shown that $\eP_0^{(s,s)}(x) \sim |x|^{-1-2s}$ as $|x| \to \infty$, matching the expected decay rate for $u \in H^s(\mathbb{R})$, and $\eP_n^{(s,s)}(x)$ decay at a faster rate for  $n \in \mathbb{N}$. Although this conjecture concerning the accurate representation of the algebraic tails is not proven in this work, we numerically verify it in \cref{sec:fractional-heat}. 
\end{remark}

\subsection{Affine transformations}
Our approach allows us to consider the unions of affine transformations of families of functions which we numerically found often improved the convergence and stability of the approximation. We outline how affine transformations are incorporated below.
\begin{definition}[Affine transformation]
\label{def:affine}
Consider an interval $I = [a,b] \subset \mathbb{R}$, $a<b$ and the affine transformation $y=2/(b-a) (x-(a+b)/2)$. We define the affine transformed function, centred at $I$, as $f^{I}(x) = f(y)$.
\end{definition}

Consider the intervals $I_k$, $k \in \{1,\dots,K\}$, and the expansion
\begin{align}
u(x) = \sum_{k=1}^K \sum_{j=0}^\infty \left[\tilde{u}_j^k \tilde{P}^{I_k, (-s,-s)}_j(x) + u_j^k Q^{I_k, (s,s)}_j(x)\right].
\label{eq:sum-expand-u}
\end{align}
We order these affine transformation in the quasimatrix ${\bm \Phi}(x)$ as follows:
\begin{align}
{\bm \Phi}(x) = 
\left( 
\tilde{P}^{I_1, (-s,-s)}_0(x) \cdots  \tilde{P}^{I_K, (-s,-s)}_0(x) 
\ | \  Q^{I_1, (s,s)}_0(x) \cdots  Q^{I_K, (s,s)}_0(x)
\ | \ \cdots
\right),
\end{align}
Thus \cref{eq:sum-expand-u} may be rewritten as $u(x) = {\bm \Phi}(x) {\mathbf u}$.

\begin{remark}
These results extend to the case where the Jacobi weight and fractional exponent do not match. The key observation is that, for $-1 < s+t < 1$,
\begin{align}
(-\Delta)^t \eP_n^{(a,s)}(x) = (-\Delta)^{t+s} Q_n^{(a,a)}(x) = \eP_n^{(a,s+t)}(x).
\end{align}
This allows one to apply the spectral method to solve general equations of the form \cref{eq:fpde}. A numerical example is provided in \cref{sec:multiple-exponents}.
\end{remark}

\subsection{Generalized Zernike polynomials}
\label{sec:zernike}
The generalized Zernike polynomials, $Z^{(b)}_{n,m,j}(x,y)$, $ b >-1$, are multivariate orthogonal polynomials in Cartesian coordinates defined on the unit disk, $\{(x,y) : x^2 + y^2 \leq 1\}$. The subscript $n$ denotes the polynomial degree and $(m,j)$ denotes the Fourier mode and sign. If $n$ is even, then $m \in \{0,2,\dots,n\}$ and, if $n$ is odd, then $m \in \{1,3,\dots,n\}$. If $m=0$, then $j=1$, otherwise $j \in \{0,1\}$. The generalized Zernike polynomials are orthogonal with respect to the weight $(1-r^2)^b$, where $r^2 = x^2 + y^2$. They may be defined via the Jacobi polynomials as:
\begin{align}
V_{m,j}(x,y) &\coloneqq r^m \sin(m\theta + j \pi/2)\\
Z^{(b)}_{n,m,j}(x,y) &\coloneqq V_{m, j}(x,y) P^{(b, m)}_{(n-m)/2}(2r^2-1).
\end{align}
We define the extended Zernike functions for $(x,y) \in \mathbb{R}^2$, $r \neq 1$, (when they exist) as
\begin{align}
\tilde{Z}^{(b,s)}_{n,m,j}(x,y) \coloneqq (-\Delta)^s [(1-r^2)^b_+ Z^{(b)}_{n,m,j}(x,y)].
\end{align}

The following theorem provides explicit expressions for $\tilde{Z}^{(s,s)}_{n,m,j}(x,y)$.
\begin{theorem}[Row A$^{**}$, Tab.~7 in \cite{Gutleb2023}]
Suppose that $b=s$, $s \in (-1,1)$ and $r^2 = x^2 + y^2$. Then, for any $n \geq 0$ and $(x,y) \in \mathbb{R}^2$, $r \neq 1$,
\begin{align}
\begin{split}
\tilde{Z}^{(s,s)}_{n,m,j}(x,y)
&=
V_{m,j}(x,y) \frac{4^s \Gamma(1 + s + n)}{\Gamma(n+1)} \\
&\times
\begin{cases} 
\frac{\Gamma \left(1+m+n+s\right) P_n^{(s,m)}(2r^2-1)}{\Gamma \left(1+n+m \right)} &r<1, \\ 
\frac{ (-1)^n \Gamma \left(1+m+n+s\right)  \, _2F_1\left(n+s+1,1+m+n+s;s+m+2 n+2;\frac{1}{r^2}\right)}{\Gamma (-n-s) \Gamma \left(s+m+2 n+2\right) r^{2(1+m+n+s)}} & r>1.
\end{cases}
\end{split}
\label{eq:extendedZernikedef}
\end{align}
\end{theorem}

\section{Frames}
\label{sec:frames}
As the family of approximating functions is not orthogonal, the expansion of a known function, $f(x)$, is nontrivial. In particular, the expansion need not be unique. For our purposes, we desire an expansion that approximates our known function to the required tolerance and the coefficients of the expansion are (relatively) small in magnitude \cite{adcock2021}.

We turn to the techniques used in frame theory \cite{Adcock2019, Adcock2020}. Essentially, we find the coefficients of the expansion that optimally interpolate the values of $f(x)$ (or a linear operator applied to $f$ if $f$ is not defined pointwise) at a set of collocation points in a least squares sense. Consider the collocation points $\vectt{x} = (x_1,\dots,x_M)$. Let ${\bm \Phi}(x)$ denote the quasimatrix for the expansion of the solution $u(x)$. Pick the number of functions in the truncation $N$.  As depicted in line 3 of \cref{alg:frame},  the least-squares matrix is given by
\begin{align}
X_{ij} = [\ell_j(\mathcal{L} {\bm \Phi})]_i, \;\; i=1,\dots,N, \; j=1\,\dots,M.
\label{def:least-squares-matrix}
\end{align}
Similarly we compute $\vectt{y}_j = \ell_j(f)$. As discussed, we choose $\ell_j(\mathcal{L} {\bm \Phi}) = (\mathcal{L} {\bm \Phi})(x_j)$, for the collocation points $\{x_j\}_{j=1}^M$,  provided  $(\mathcal{L} {\bm \Phi})(x_j)$ and $f(x_j)$ are well-defined. We then solve the following least-squares problem for the expansion coefficients ${\mathbf u}$:
\begin{align}
\min_{{\mathbf u} \in \mathbb{C}^N} \|X{\mathbf u} - \vectt{y}\|_{\ell^2}, 
\label{eq:least-squares}
\end{align}
so that $f \approx \mathcal{L} {\bm \Phi} {\mathbf u}$ and, therefore, $u \approx {\bm \Phi}{\mathbf u}$ (as discussed in \cref{sec:spectral-method}). In a frame setting, the least-squares matrix is often ill-conditioned for increasing $M$ and $N$. However, in practice we recover suitable least squares solutions if we sufficiently oversample the collocation points and use a truncated SVD solver. The remainder of this section will focus on explaining this phenomenon.

Consider the SVD factorization $X = U \Sigma V^\top$ with the convention where $\Sigma$ is a square diagonal matrix with the singular values arranged in decreasing magnitude along the diagonal. Given a precision tolerance $\epsilon$, a truncated SVD solver finds the first singular value such that $\Sigma_{i+1,i+1} < \epsilon < \Sigma_{i,i}$. Then, the submatrix $\Sigma_\epsilon \in \mathbb{R}^{i \times i}$, $[\Sigma_\epsilon]_{j,j}= \Sigma_{j,j}$, $j \in \{1,\dots,i\}$ is extracted. The $\epsilon$-truncated SVD projection of \cref{eq:least-squares} is given by
\begin{align}
{\mathbf u}_\epsilon = V_\epsilon \Sigma^{-1}_\epsilon U_\epsilon^\top {\mathbf y},
\end{align}
where $V_\epsilon$ and $U_\epsilon$ are the first $i$ columns of $V$ and $U$, respectively. 

\begin{remark}[Truncated SVD cost]
With $M>N$ the cost of computing the SVD of $X$ scales as $\mathcal{O}(MN^2)$. In the examples found in \cref{sec:examples}, we found that this cost is negligible but since our analytical results generalize to two and three dimensions, the cost of the factorization may eventually prove prohibitive. If required, then for suitable frames, the solve may be implemented in $\mathcal{O}(N\log^2 N)$ operations via the $AZ$-algorithm \cite{Coppe2020} and randomized linear algebra solvers for least-squares \cite{Halko2011, Meier2022, Meier2023}.
\end{remark}

\begin{remark}[Evaluation of the least-squares matrix]
Another potential computational bottleneck is the assembly of the least-squares matrix $X$ as defined in \cref{def:least-squares-matrix}, particularly if we are required to evaluate many hypergeometric functions at many collocation points. Although a direct evaluation of many ${_2}F_1$ functions may be slow, we note that the extended Jacobi functions satisfy the same three-term recurrence as their Jacobi polynomial counterparts. Thus one may use the forward recurrence for efficient evaluation of the extended Jacobi functions on the interval $[-1,1]$ and (F.~W.~J) Olver's and Miller's algorithm for evaluation off the interval $[-1,1]$, cf.~\cite[Sec.~2.3]{Gautschi2004}, \cite[Sec.~3.6]{OlverNIST} and \cite[App.~B]{Olver2019}. Alternatively, off the interval one can use a continued fraction approach advocated by Gautschi \cite{Gautschi1981}.
\end{remark}

We now introduce the notion of a \emph{frame}.
\begin{definition}[Frame]
Consider a Hilbert space $(H, (\cdot,\cdot)_H)$. An indexed family of functions $\{\phi_n\} \in H$ is called a \emph{frame} for $H$ if there exist constants $0 < c \leq C < \infty$ such that
\begin{align}
c\| v \|^2_H \leq \sum_n |( v, \phi_n )_H |^2 \leq C \|v \|^2_H \;\; \text{for all} \; \; v \in H.
\label{ineq:frame-condition}
\end{align}
\end{definition}

As the family of functions for the expansion of the right-hand side will vary according to the choice of \cref{eq:fpde}, it is helpful to construct a framework whereby if one shows that the family of functions for the solution expansion is a frame (which remains fixed), then the induced family of functions for the right-hand  side  is automatically a frame provided that the operator $\mathcal{L}$ in \cref{eq:fpde} operator is well-behaved. 
\begin{theorem}[Frames on the dual space]
\label{lem:range}
Suppose that the family of functions $\Phi = \{\phi_n\}$ is a frame on the Hilbert space $H$. Consider the bounded linear operator $\mathcal{L} : H \to  H^*$ where $H^*$ is the dual space of $H$. Moreover, assume that the adjoint operator   $\mathcal{L}^* : H \to H^*$ is bounded above and below with constants $M_b, M_c > 0$, respectively, i.e.~for all $u, v \in H$,
\begin{align}
\langle \mathcal{L} u, v \rangle_{H^*, H} = \langle \mathcal{L}^* v, u \rangle_{H^*, H} \;\;
\text{and} \;\;
M_c \| u \|_H \leq \| \mathcal{L}^* u \|_{H^*} \leq M_b \| u \|_{H}.
\end{align} 
Then $\mathcal{L}{\Phi} = \{\mathcal{L} \phi_n\}$ is a frame on the Hilbert space $H^*$.
\end{theorem} 
\begin{proof}
Let $\psi_n = \mathcal{L}\phi_n$, for any $n \in \mathbb{N}_0$. Let $R : H \to H^*$ denote the Riesz isomorphism and $R^{-1} : H^* \to H$ its inverse \cite[Ch.~D.3, Th.~2]{Evans2010}. Now consider any $f \in H^*$. The lower bound is derived as follows: 
\begin{align}
\begin{split}
&c M_c^2\| f \|^2_{H^*} 	
= c M_c^2 \| R^{-1} f \|^2_H
\leq c \| \mathcal{L}^* R^{-1} f \|^2_{H^*}
 =  c \| R^{-1} \mathcal{L}^* R^{-1} f \|^2_{H}\\
& \indent \leq  \sum_{n} ( R^{-1} \mathcal{L}^* R^{-1} f, \phi_n)^2_{H}
 = \sum_{n} \langle \mathcal{L}^* R^{-1} f, \phi_n\rangle^2_{H^*,H}
 =  \sum_{n} \langle  R^{-1} f, \mathcal{L} \phi_n\rangle^2_{H,H^*}
=\sum_{n} (  f, \psi_n )^2_{H^*}.
\label{eq:dual-frame1}
\end{split}
\end{align} 
Here $c$ is the lower bound for the frame condition of $\Phi$.  The first inequality follows from the lower-boundedness of the adjoint operator $\mathcal{L}^*$ and the second inequality follows from the frame condition of $\Phi$. The final equality follows by the definition of $\psi_n$.

Similarly, let $C$ denote the upper bound for the frame condition of $\Phi$. Then
\begin{align}
\begin{split}
&\sum_{n} (  f, \psi_n )^2_{H^*}
= \sum_{n} ( R^{-1} \mathcal{L}^* R^{-1} f, \phi_n)^2_{H} \\
& \indent \leq C \| R^{-1} \mathcal{L}^* R^{-1} f \|^2_{H} 
= C \| \mathcal{L}^* R^{-1} f \|^2_{H^*} 
\leq C M_b^2 \| R^{-1} f \|^2_{H} 
= C M_b^2 \|  f \|^2_{H^*} 
\end{split}
\end{align}
The first equality follows by backtracking the equalities from the right-hand side in \cref{eq:dual-frame1}. The first inequality follows from the frame property of $\Phi$ and the second inequality by the upper-boundedness assumption on $\mathcal{L}^*$. 
\end{proof} 

\subsection{Construction of frames for $\mathcal{L} = (\mathcal{I} + (-\Delta)^s)$}

In this subsection, we show that a family of functions consisting of weighted Jacobi polynomials and their fractional Laplacian counterparts ${\bm \Phi}(x)$ are a frame on a weighted Hilbert space. Then, by utilizing \cref{lem:range}, we conclude that $(\mathcal{I} + (-\Delta)^s){\bm \Phi}(x)$ is a frame on the dual space.

\begin{definition}[Weighted Lebesgue space]
Let $w :  [a,b]  \to [0,\infty)$ denote a Muckenhoupt weight \cite{muckenhoupt1972}. Then we denote the weighted Lebesgue space by $L^p_w(a,b)$ when equipped with the norm $\|f\|_{L^p_w(a,b)} \coloneqq (\int_{a}^b |f|^p w \, \dx)^{1/p}$. If $p=2$, then $L^2_w(a,b)$ is a Hilbert space equipped with the inner product $(f,g)_{L^2_w(a,b)} \coloneqq \int_{a}^b f g w \, \dx$.
\end{definition}

\begin{definition}[Hilbert space $H^s_w(\mathbb{R})$]
\label{def:Hsw}
 Let $\Omega \coloneqq \mathrm{int} \, (\supp \, w)$. Here $\mathrm{int}$ denotes the interior of a set and $\supp$ denotes the support of a function.  Then $H^s_w(\mathbb{R})$ denotes the Hilbert space 
\begin{align}
H^s_w(\mathbb{R}) \coloneqq
\{ u \in L^2_w( \Omega ) : \supp(u) \subseteq \supp(w), \; (-\Delta)^{s/2} u \in L^2(\mathbb{R})\},
\end{align}
equipped with the inner-product $$( u, v )_{H^s_w(\mathbb{R})} \coloneqq ( u, v )_{L^2_w( \Omega )} +  ((-\Delta)^{s/2} u,  (-\Delta)^{s/2} v )_{L^2(\mathbb{R})}.$$ 
\end{definition}

\begin{remark}
When considering nonlocal operators defined on $\mathbb{R}$, such as $(-\Delta)^s$, we note that even if the support of a function $u$ is compactly contained in $\mathbb{R}$, this does not mean $(-\Delta)^s u$ is compactly supported in $\mathbb{R}$. This is the motivation behind considering the inner-product defined on all of $\mathbb{R}$ in \cref{def:Hsw}.
\end{remark}

For the remainder of this section we define, for $s \in (0,1)$, $w_s(x) \coloneqq (1-x^2)_+^{-s}$. Moreover, as defined in \cref{thm:fractionalwholespaceultra}, let $c_{s,n} \in \mathbb{R}$, $n \in \mathbb{N}_0$, denote the constants such that, for all $x \in (-1,1)$, $\eP_n^{(s,s)}(x) = c_{s,n} P_n^{(s,s)}(x)$.

The following lemma proves orthogonality of weighted and extended Jacobi polynomials. This is heavily utilized for the frame result in \cref{prop:frame}. A similar result may be found in \cite[Prop.~3.6]{Papadopoulos2022e}. 
\begin{lemma}[Orthogonality]
\label{lem:orthogonality}
$\{ Q_n^{(s,s)} \}_{n \in \mathbb{N}_0}$ and $\{ \tilde{P}_n^{(-s,-s)} \}_{n \in \mathbb{N}_0}$ satisfy
\begin{align}
( \tilde{P}_n^{(-s,-s)}, \tilde{P}_m^{(-s,-s)} )_{H^s_{w_s}(\mathbb{R})}  = \tilde{C}_{n,m} \delta_{n,m} \;\;\text{and} \;\;
( Q_n^{(s,s)}, Q_m^{(s,s)} )_{H^s_{w_s}(\mathbb{R})}  = C_{n,m} \delta_{n,m},
\end{align}
where $\delta_{nm}$ denotes the Kronecker delta and $C_{n,m}, \tilde{C}_{n,m} \in \mathbb{R}$. 
\end{lemma}
\begin{proof}
Let  $c_{s,n} \in \mathbb{R}$, $n \in \mathbb{N}_0$, denote the constants such that, for all $x \in (-1,1)$, $\eP_n^{(s,s)}(x) = c_{s,n} P_n^{(s,s)}(x)$. Then,
\begin{align}
\begin{split}
( Q_n^{(s,s)}, Q_m^{(s,s)} )_{H^s_{w_s}(\mathbb{R})}
& =  ( Q_n^{(s,s)}, Q_m^{(s,s)} )_{L^2_{w_s}(-1,1)}+ ( (-\Delta)^{s/2} Q_n^{(s,s)}, (-\Delta)^{s/2} Q_m^{(s,s)} )_{L^2(\mathbb{R})} \\
& = ( Q_n^{(s,s)}, {P}_m^{(s,s)} )_{L^2(-1,1)}+ \int_{\mathbb{R}}  Q_n^{(s,s)} \cdot (-\Delta)^{s} Q_m^{(s,s)} \, \dx \\
&  = ( Q_n^{(s,s)}, {P}_m^{(s,s)} )_{L^2(-1,1)} + ( Q_n^{(s,s)}, c_{s,m} P_m^{(s,s)} )_{L^2(-1,1)} = {C}_{n,m} \delta_{n,m}.
\end{split}
\end{align}
The result for $\{ \eP_n^{(-s,-s)} \}_{n \in \mathbb{N}_0}$ follows similarly.
\end{proof}

We define $\{\hat{Q}_n^{(s,s)}\}$ and $\{\hat{P}_n^{(-s,-s)}\}$ as the orthonormalized families of functions, $\{Q_n^{(s,s)}\}$ and $\{\eP_n^{(-s,-s)}\}$, respectively, with respect to the $H^s_{w_s}(\mathbb{R})$ inner-product, i.e.
\begin{align}
\hat{Q}_n^{(s,s)}(x) \coloneqq \frac{{Q}_n^{(s,s)}(x)}{\| {Q}^{(s,s)}_n \|_{H^s_{w_s}(\mathbb{R})}} \;\;
\text{and} \;\;
\hat{P}_n^{(-s,-s)}(x) \coloneqq \frac{\tilde{P}_n^{(-s,-s)}(x)}{\| \tilde{P}^{(-s,-s)}_n \|_{H^s_{w_s}(\mathbb{R})}}.
\end{align}

We now state and prove two lemmas which aid with deriving the result in \cref{prop:frame}.
\begin{lemma}
\label{lem:a}
Consider any $u \in H^s_{w_s}(\mathbb{R})$ with $s \in (0,1)$. Then
\begin{align}
(u, \hat{Q}_n^{(s,s)})_{H^s_{w_s}(\mathbb{R})} = (1+c_{s,n}) (u, \hat{Q}_n^{(s,s)})_{L^2_{w_s}(-1,1)}.
\end{align}
\end{lemma}
\begin{proof}
First recall that, for $x \in \mathbb{R}  \backslash \{-1, 1\}$,
\begin{align}
(-\Delta)^s \hat{Q}^{(s,s)}_n(x) 
= \frac{(-\Delta)^s {Q}^{(s,s)}_n(x)}{\| {Q}^{(s,s)}_n \|_{H^s_{w_s}(\mathbb{R})}}
=  \frac{\tilde{P}_n^{(s,s)}}{\|{Q}_n^{(s,s)}\|_{H^s_{w_s}(\mathbb{R})}}.
\end{align}
Then we note that
\begin{align}
\begin{split}
(u, \hat{Q}_n^{(s,s)})_{H^s_{w_s}(\mathbb{R})}  
&= (u, \hat{Q}_n^{(s,s)})_{L^2_{w_s}(-1,1)} + ((-\Delta)^{s/2} u, (-\Delta)^{s/2}\hat{Q}_n^{(s,s)})_{L^2(\mathbb{R})}\\
&= (u, \hat{Q}_n^{(s,s)})_{L^2_{w_s}(-1,1)} + \int_{\mathbb{R}} u \cdot (-\Delta)^{s} \hat{Q}_n^{(s,s)}\, \dx\\
& = (u, \hat{Q}_n^{(s,s)})_{L^2_{w_s}(-1,1)} + \int_{\mathbb{R}} u \cdot  \frac{{\tilde P}_n^{(s,s)}}{\|{Q}_n^{(s,s)}\|_{H^s_{w_s}(\mathbb{R})}} \, \dx\\
& = (u, \hat{Q}_n^{(s,s)})_{L^2_{w_s}(-1,1)} + c_{s,n} \int_{-1}^1 u \cdot  \frac{{P}_n^{(s,s)}}{\|{Q}_n^{(s,s)}\|_{H^s_{w_s}(\mathbb{R})}} \, \dx\\
& = (u, \hat{Q}_n^{(s,s)})_{L^2_{w_s}(-1,1)} + c_{s,n} \int_{-1}^1 u \cdot  \frac{(1-x^2)^s {P}_n^{(s,s)}}{\|{Q}_n^{(s,s)}\|_{H^s_{w_s}(\mathbb{R})}} w_s(x) \, \dx\\
& = (1+c_{s,n}) (u, \hat{Q}_n^{(s,s)})_{L^2_{w_s}(-1,1)}.
\end{split}
\end{align}
The fourth equality follows from the fact that $\supp(u) \subseteq \supp(w_s) = [-1,1]$. 
\end{proof}

\begin{lemma}
\label{lem:b}
Consider any $u \in H^s_{w_s}(\mathbb{R})$ with $s \in (0,1)$. Define the constants, for $n \in \mathbb{N}_0$, 
\[
b_n \coloneqq \frac{(u, \hat{Q}_n^{(s,s)})_{L^2_{w_s}(-1,1)}}{\|\hat{Q}_n^{(s,s)}\|^2_{L^2_{w_s}(-1,1)}}.
\]
Then, for any $N \in \mathbb{N}_0$,
\begin{align}
\| \sum_{n=0}^N b_n \hat{Q}_n^{(s,s)}\|^2_{H^s_{w_s}(\mathbb{R})} = \sum_{n=0}^N (1+c_{s,n})^2 (u, \hat{Q}_n^{(s,s)})^2_{L^2_{w_s}(-1,1)}.
\end{align}
\end{lemma}
\begin{proof}
Note that
\begin{align}
\begin{split}
\| \sum_{n=0}^N b_n \hat{Q}_n^{(s,s)}\|^2_{H^s_{w_s}(\mathbb{R})} 
 = (\sum_{n=0}^N b_n \hat{Q}_n^{(s,s)}, \sum_{m=0}^N b_m \hat{Q}_m^{(s,s)})_{H^s_{w_s}(\mathbb{R})} 
= \sum_{n=0}^N b_n^2 
= \sum_{n=0}^N \frac{(u, \hat{Q}_n^{(s,s)})^2_{L^2_{w_s}(-1,1)}}{\|\hat{Q}_n^{(s,s)}\|^4_{L^2_{w_s}(-1,1)}}.
\label{eq:lem1}
\end{split}
\end{align}
Moreover,
\begin{align}
\begin{split}
1 
&= \| \hat{Q}^{(s,s)}_n \|^2_{H^s_{w_s}(\mathbb{R})} \\
& = \|\hat{Q}^{(s,s)}_n\|^2_{L^2_{w_s}(-1,1)} + ((-\Delta)^{s/2} \hat{Q}_n^{(s,s)}, (-\Delta)^{s/2}\hat{Q}_n^{(s,s)})_{L^2(\mathbb{R})} \\
&=  \|\hat{Q}^{(s,s)}_n\|^2_{L^2_{w_s}(-1,1)} + \int_{\mathbb{R}} \hat{Q}_n^{(s,s)} \cdot (-\Delta)^{s} \hat{Q}_n^{(s,s)}\, \dx\\
&=  \|\hat{Q}^{(s,s)}_n\|^2_{L^2_{w_s}(-1,1)} + \int_{\mathbb{R}} \hat{Q}_n^{(s,s)} \cdot  \frac{{\tilde P}_n^{(s,s)}}{\|{Q}_n^{(s,s)}\|_{H^s_{w_s}(\mathbb{R})}}\, \dx\\
&=\|\hat{Q}^{(s,s)}_n\|^2_{L^2_{w_s}(-1,1)} + c_{s,n} \int_{-1}^1 \hat{Q}^{(s,s)}_n \cdot  \frac{(1-x^2)^s {P}_n^{(s,s)}}{\|{Q}_n^{(s,s)}\|_{H^s_{w_s}(\mathbb{R})}} w_s(x) \, \dx\\
&= (1+c_{s,n})  \|\hat{Q}^{(s,s)}_n\|^2_{L^2_{w_s}(-1,1)}.
\label{eq:lem2}
\end{split}
\end{align}
\cref{eq:lem2} implies that $ \|\hat{Q}^{(s,s)}_n\|^4_{L^2_{w_s}(-1,1)} = (1+c_{s,n})^{-2}$. Plugging this into \cref{eq:lem1}, we arrive at the result.
\end{proof}

\begin{theorem}[Frame on $H^s_{w_s}(\mathbb{R})$]
\label{prop:frame}
The space spanned by $\{\hat{Q}_n^{(s,s)}\}_{n=0}^\infty \cup \{\hat{P}_n^{(-s,-s)}\}_{n=0}^\infty$ is a frame on $H^s_{w_s}(\mathbb{R})$ for any $s \in (0,1)$.
\end{theorem}
\begin{proof}
Consider any $u \in H^s_{w_s}(\mathbb{R})$ with $s \in (0,1)$. Recall that we must show the existence of constants $c, C>0$ such that the frame condition \cref{ineq:frame-condition} is satisfied.

\noindent
\textbf{Upper bound.} Since $\{\hat{P}^{(-s,-s)}_n\}$ and $\{\hat{Q}^{(s,s)}_n\}$ are orthonormal sequences in $H^s_{w_s}(\mathbb{R})$, then Bessel's inequality \cite[Ch.~8, Sec.~4--5]{Yosida1980} implies that
\begin{align}
\begin{split}
\sum_{n=0}^\infty \left[ (u, \hat{P}_n^{(-s,-s)})^2_{H^s_{w_s}(\mathbb{R})} + (u, \hat{Q}_n^{(s,s)})^2_{H^s_{w_s}(\mathbb{R})} \right]
\leq 
2 \| u \|^2_{H^s_{w_s}(\mathbb{R})}.
\end{split}
\end{align}

\noindent
\textbf{Lower bound.}
By the Corollary of Theorem 1 in \cite{muckenhoupt1969}, we have that
\begin{align}
\lim_{N \to \infty} \| u - \sum_{n=0}^N b_n \hat{Q}_n^{(s,s)} \|_{L^2_{w_s}(-1,1)} = 0 \;\;
\text{where} \;\;
b_n = \frac{(u, \hat{Q}_n^{(s,s)})_{L^2_{w_s}(-1,1)}}{ \|\hat{Q}_n^{(s,s)}\|^2_{L^2_{w_s}(-1,1)}}.
\label{eq:prop:frame1}
\end{align}
Moreover, for any $N \in \mathbb{N}_0$, we observe that
\begin{align}
\begin{split}
\| \sum_{n=0}^N b_n \hat{Q}_n^{(s,s)}\|^2_{H^s_{w_s}(\mathbb{R})} 
&= \sum_{n=0}^N (1+c_{s,n})^2 (u, \hat{Q}_n^{(s,s)})^2_{L^2_{w_s}(-1,1)}\\
& = \sum_{n=0}^N (u, \hat{Q}_n^{(s,s)})^2_{H^s_{w_s}(\mathbb{R})} \\
& \leq \sum_{n=0}^\infty (u, \hat{Q}_n^{(s,s)})^2_{H^s_{w_s}(\mathbb{R})} 
\leq \| u \|^2_{H^s_{w_s}(\mathbb{R})},
\label{eq:prop:frame2}
\end{split}
\end{align}
where the first equality follows from \cref{lem:b}, the second equality from \cref{lem:a} and the final inequality by an application of Bessel's inequality. Thus the sequence $\sum_{n=0}^N b_n \hat{Q}_n^{(s,s)}$ for $N \in \mathbb{N}_0$ is uniformly bounded in $H^s_{w_s}(\mathbb{R})$. As $H^s_{w_s}(\mathbb{R})$ is a Hilbert space and therefore, weakly sequentially compact, we may extract a subsequence (not relabeled) that converges weakly in $H^s_{w_s}(\mathbb{R})$ to a limit. As weak limits are unique, we may identify this limit as $u$ via \cref{eq:prop:frame1}. Since a squared-norm is convex and bounded below, it is weakly lower-semicontinuous \cite[Ch.~8.2, Th.~1]{Evans2010}. Hence,
\begin{align}
\begin{split}
\| u \|^2_{H^s_{w_s}(\mathbb{R})} 
\leq \liminf_{N \to \infty} \| \sum_{n=0}^N b_n \hat{Q}_n^{(s,s)}\|^2_{H^s_{w_s}(\mathbb{R})} 
= \liminf_{N \to \infty}  \sum_{n=0}^N (u, \hat{Q}_n^{(s,s)})^2_{H^s_{w_s}(\mathbb{R})} \leq \sum_{n=0}^\infty (u, \hat{Q}_n^{(s,s)})^2_{H^s_{w_s}(\mathbb{R})},
\label{eq:prop:frame3}
\end{split}
\end{align}
where the equality follows from the second equality in \cref{eq:prop:frame2}. From \cref{eq:prop:frame3} we deduce the frame lower bound condition
\begin{align}
\| u \|^2_{H^s_{w_s}(\mathbb{R})}  \leq \sum_{n=0}^\infty \left[ (u, \hat{Q}_n^{(s,s)})^2_{H^s_{w_s}(\mathbb{R})} + (u, \hat{P}_n^{(-s,-s)})^2_{H^s_{w_s}(\mathbb{R})} \right].
\end{align}
\end{proof}

A basic building block for a frame suitable for expanding the solution of $(\mathcal{I} + (-\Delta)^s)u = f$ is $\Phi = \{\hat{Q}_n^{(s,s)}\}_{n=0}^\infty \cup \{\hat{P}_n^{(-s,-s)}\}_{n=0}^\infty$. The operator $\mathcal{L} = (-\Delta)^s + \mathcal{I}$ implies we expand the right-hand side $f(x)$ in  $\Psi =\{(-\Delta)^s \hat{Q}^{(s,s)}_n +  \hat{Q}^{(s,s)}_n \}_{n \in \mathbb{N}_0} \cup \{ (-\Delta)^s \hat{P}^{(-s,-s)}_n +  \hat{P}^{(-s,-s)}_n \}_{n \in \mathbb{N}_0}$. In the next corollary we show that by utilizing \cref{lem:range} and \cref{prop:frame}, we may deduce that the space spanned by $\Psi$ is also a frame.

\begin{corollary}[Frame on $H^s_{w_s}(\mathbb{R})^*$]
\label{cor:rhs-frame}
The space spanned by $\Psi =\{(-\Delta)^s \hat{Q}^{(s,s)}_n +  \hat{Q}^{(s,s)}_n \}_{n \in \mathbb{N}_0} \cup \{ (-\Delta)^s \hat{P}^{(-s,-s)}_n +  \hat{P}^{(-s,-s)}_n \}_{n \in \mathbb{N}_0}$ is a frame on $H^s_{w_s}(\mathbb{R})^*$.
\end{corollary} 
\begin{proof}
Let $\mathcal{L} = (-\Delta)^s + \mathcal{I}$. In \cref{prop:frame} it was shown that $\Phi = \mathcal{L}^{-1} \Psi$ is a frame on $H = H^s_{w_s}(\mathbb{R})$. By \cref{lem:range}, it is sufficient to show that  $\mathcal{L}^* u \in H^*$ is bounded above and below for any $u \in H$. A calculation reveals
\begin{align}
\| \mathcal{L}^* u \|_{H^*} = \sup_{v \in H} \frac{ | (u,v)_{H}|}{\| v\|_{H}}.
\label{cor:dual-frame1}
\end{align}
By picking $v = u$ on the right hand side of \cref{cor:dual-frame1}, we conclude the lower bound with constant one. Similarly an application of the Cauchy--Schwarz inequality provides the upper bound with constant one.
\end{proof}

We conclude this section with the following a priori estimate which constitutes the main theorem of this work.
\begin{theorem}[A priori estimate]
\label{th:apriori}
Consider equation \cref{eq:fpde} and suppose that the conditions of \cref{lem:range} hold. Let ${\bm \Phi}$ denote the quasimatrix of a frame for the Hilbert space $H = H^s_w(\mathbb{R})$ with $H^s_w(\mathbb{R})$ as defined in \cref{def:Hsw}. Consider the truncated quasimatrices ${\bm \Phi}_N = (\phi_1 \; \phi_2 \; \cdots \phi_N)$ and ${\bm \Psi}_N = \mathcal{L}{\bm \Phi}_N$ and choose the bounded linear functionals $\ell_j$, $j=1,\dots,M$, in \cref{alg:frame} such that $\ell_j(f)$ are well-defined for a given $f \in H^s_w(\mathbb{R})^* \subseteq H^{-s}(\mathbb{R})$. Suppose that ${\mathbf u}_N \in \mathbb{R}^N$ is computed via \cref{alg:frame}. Then, there exist constants $\lambda, \kappa > 0$ such that
\begin{align}
\begin{split}
\| u - {\bm \Phi}_N {\mathbf u}_N \|_{H}
 \leq \| \mathcal{L}^{-1}\|_{\mathcal{B}(H^*, H)}
 \inf_{{\mathbf v} \in \mathbb{C}^N} 
\left\{
\| f -  {\bm \Psi}_N {\mathbf v} \|_{H^*} + \kappa \| f -  {\bm \Psi}_N {\mathbf v} \|_M + \epsilon \lambda \| {\mathbf v}\|_{\ell^2}
 \right\},
\end{split}
\end{align}
where $\| v \|^2_M \coloneqq \sum_{j=1}^M \ell_j(v)^2$. 
\end{theorem}

\begin{remark}
The constants $\lambda$ and $\kappa$ are dependent on truncated SVD tolerance $\epsilon$, the collocation points, and $N$ as well as the problem and the frame. Ideally their magnitudes are on the order of one as $N \to \infty$ for careful choices of the collocation points, in which case we are guaranteed to reach an accuracy of $\epsilon$ if $N$ is taken sufficiently large. For a thorough discussion of the constants and their behaviour, we refer the reader to \cite{Adcock2020}.
\end{remark}

\begin{proof}[Proof of \cref{th:apriori}]
First note that
\begin{align}
\| u - {\bm \Phi}_N {\mathbf u}_N \|_{H} 
= \| \mathcal{L}^{-1} f - \mathcal{L}^{-1} {\bm \Psi}_N {\mathbf u}_N \|_{H}
\leq  \| \mathcal{L}^{-1}\|_{\mathcal{B}(H^*, H)} \|  f - {\bm \Psi}_N {\mathbf u}_N \|_{H^*}.
\end{align}
By \cref{lem:range}, $\Psi = \mathcal{L}\Phi$ is a frame on $H^*$. Moreover, by definition, ${\mathbf u}_N$ is the vector obtained by conducting an $\epsilon$-truncated SVD projection for frame expansion of $f$, as described in \cref{eq:least-squares}. Thus, by leveraging Theorem 3.7 in \cite{Adcock2020}, we conclude the result.
\end{proof}

\begin{remark}[Convergence rates]
Suppose that the right-hand side $f$ in \cref{eq:fpde} is smooth and has an asymptotic algebraic decay (or faster) when $|x| \to \infty$. As a rule of thumb, if one constructs a solution frame $\Phi$ containing extended and weighted Jacobi or Zernike functions, centred on various affine-transformed intervals or balls, that cover the support of $f$ prior to where the asymptotic decay dominates, then \cref{th:apriori} indicates we expect \emph{spectral} convergence to the solution of \cref{eq:fpde}. The rough argument is that the induced frame $\Psi$ for $f$ includes functions that are piecewise polynomials prior to the asymptotic decay, with increasing degree as $N \to \infty$. Moreover, $\Psi$ also includes functions that can match an asymptotic algebraic (or faster) decay. Proving such a conjecture is beyond the scope of this work, however, we numerically verify this observation in \cref{sec:examples}.
\end{remark}

\section{Time-dependent problems}
\label{sec:time-dependent}
Implementing solvers for fractional-in-space time dependent problems is notoriously difficult due to the slow decay of the solutions and the accumulation of errors in the coefficients of the expansion. If the domain is truncated then the artificial boundary layers will eventually dominate the error in the solution. Similarly, although many time-stepping schemes are initially stable, the coefficients of the spatial expansion often become increasingly larger in magnitude leading to floating point cancellation error and a breakdown in the solution for large time $t$. 

Our spectral method may be coupled with any Runge--Kutta (RK) method allowing for an arbitrary order method in time and spectral in space. Consider the fractional-in-space time-dependent equation:
\begin{align}
\partial_t u + F(t, u) = 0,\; u(x,0) = u_0(x), \label{eq:arb}
\end{align}
where $u(\cdot, t) \in H^s(\mathbb{R})$ for a.e.~$t \geq 0$. For instance, in the fractional heat equation, we have $F(t,u) =  (-\Delta)^s u(x,t)$ such that \cref{eq:arb} becomes
\begin{align}
\partial_t  u(x,t) + (-\Delta)^s u(x,t) = 0,\; u(x,0) = u_0(x). \label{eq:heat}
\end{align}
Now given $u(x, t_j) = v_j(x)$ and some $t_{j+1} = t_j + \delta t$, an $m$-stage RK method approximates $u(x, t_{j+1})$ with
\begin{align}
u(x, t_{j+1}) \approx v_{j+1}(x) = v_j(x) + \delta t \sum_{i=1}^m b_i k_i(x),
\label{eq:rk2}
\end{align}
where for all $1 \leq i \leq m$ the stages $k_i$ satisfy
\begin{align}
k_i + F\left(t + c_i \delta t, u + \delta t \sum_{l=1}^m A_{il} k_l \right) = 0.
\label{eq:rk}
\end{align}
The coefficients $b_i$, $c_i$ and $A_{il}$ are chosen so that the resulting method has the required degree of accuracy as well as other favourable properties e.g.~stability. \cref{eq:rk} is of the form \cref{eq:fpde}. Thus we may use the frame expansion to discretize in space as described in \cref{alg:rk} for \cref{eq:heat}.

\begin{algorithm}[ht]
\caption{$m$-stage Runge--Kutta method for the fractional heat equation \cref{eq:heat}.}
\label{alg:rk}
\begin{algorithmic}[1]
\Input{
$A \in \mathbb{R}^{m \times m}$, $b_i$, $i \in \{1,\dots,m\}$. \Comment{$m$-stage Runge--Kutta coefficients.}\\
${\mathbf u}_0$ \Comment{Coefficients for the initial state $u_0(\cdot)$.}\\ 
$X$, $X_*$ {\Comment{Least-squares matrices for ${\bm \Phi}_N(x)$ and $ (-\Delta)^s {\bm \Phi}_N(x)$.}}\\
 $\delta t$, J \Comment{Time step and maximum time step iterations.}
\vspace{2mm}
}
\State{Assemble $X_A = I_m \otimes X + \delta t (A \otimes X_*)$, where $\otimes$ denotes the Kronecker product and $I_m$ is the $m \times m$ identity matrix.}
\For{$j = 1, 2, \dots, J$}
\State{${\mathbf y} \gets X_* {\mathbf u}_{j-1}$.}
\State{Via an  $\epsilon$-truncated SVD projection, compute \newline
\hspace*{6mm}${\bm k}_A \approx \mathrm{argmin}_{{\mathbf v} \in \mathbb{C}^{Nm}}\| X_A {\mathbf v} - (\underbrace{{\mathbf y}^\top \cdots {\mathbf y}^\top}_{m \, \text{repeats}})^\top \|_{\ell^2}$.}
\State{Extract the $m$ vectors $\{{\bm k}_i\}_{i = 1}^m$, ${\bm k}_A = ({\bm k}_1^\top \, \cdots {\bm k}_m^\top)^\top$.}
\State{${\mathbf u}_j \gets {\mathbf u}_{j-1} + \delta t \sum_{i=1}^m b_i {\bm k}_i$.}
\EndFor
\end{algorithmic}
\end{algorithm}

By considering \cref{eq:heat}, setting $m=A_{11}=c_1=b_1 = 1$ in \cref{eq:rk2}--\cref{eq:rk}, and multiplying the result by $\delta t$, one recovers the implicit Euler discretization of \cref{eq:heat}:
\begin{align}
v_{j+1}(x) + \delta t (-\Delta)^s v_{j+1}(x) = v_{j}(x). 
\label{eq:heat2}
\end{align}
Consider the Hilbert space $H$ such that $(-\Delta)^s : H \to H^*$ and a truncated quasimatrix frame ${\bm \Phi}_N = (\phi_1,\dots,\phi_N)$ for $H$. Let ${\bm \Phi}_N^* = (\mathcal{I} + \delta t (-\Delta)^s) {\bm \Phi}_N$. Then, the implicit Euler discretization of \cref{eq:heat} amounts to the following: given the coefficient vector, ${\mathbf u}_j$, at time-step $j$, the next time-step, ${\mathbf u}_{j+1}$ satisfies
\begin{align}
{\mathbf u}_{j+1} = {\mathbf u}^*_{j} \;\; \text{where} \;\; {\bm \Phi}^*_N(x) {\mathbf u}^*_{j} = (\mathcal{I} + \delta t (-\Delta)^s) {\mathbf u}^*_{j}   \approx {\bm \Phi}_N(x) {\mathbf u}_{j}.
\label{eq:timestep:1}
\end{align}
In practice we compute ${\mathbf u}^*_{j}$ by an $\epsilon$-truncated SVD projection. Hence, for some $h_j > 0$, the approximate equality in \cref{eq:timestep:1} satisfies
\begin{align}
\| {\bm \Phi}^*_N(x) {\mathbf u}^*_{j} - {\bm \Phi}_N(x) {\mathbf u}_{j} \|_{H^*} \leq h_j \|  {\bm \Phi}_N(x) {\mathbf u}_{j} \|_{H^*}.
\label{eq:timestep}
\end{align}

In the remainder of this section we show that an implicit Euler discretization in time coupled with a frame expansion  in space, as described in \cref{sec:spectral-method}, of the fractional heat equation, results in the expected convergence, i.e.~$\mathcal{O}(\delta t + h)$ where $h  =\max_j h_j$. This result is confirmed in the numerical results in \cref{sec:examples} and typically, for sufficiently smooth data and large $M$, we have $h \leq \max \{ c_1\epsilon, \mathrm{exp}(-c_2N) \}$ for some $c_1, c_2 > 0$.

\begin{lemma}
\label{lem:bounded}
Consider the Hilbert space $H^s_\omega(\mathbb{R})$, where $\omega(x) = 1$ in $\Omega \subseteq \mathbb{R}$, for an arbitrary open interval $\Omega$, and $\omega(x) = 0$ in $\mathbb{R} \backslash \Omega$. Let $\mathcal{L} = \mathcal{I} + (-\Delta)^s$. Fix any $u^* \in H^s_\omega(\mathbb{R})$ and also denote by $u^* \in H^s_\omega(\mathbb{R})^*$ the embedding of $u^*$ into $H^s_\omega(\mathbb{R})^*$. Then, we have that
\begin{align}
\| u^* \|_{H^s_\omega(\mathbb{R})^*} \leq \| \mathcal{L}^{-1} \|_{\mathcal{B}(H^s_\omega(\mathbb{R}),H^s_\omega(\mathbb{R}))} \| u^* \|_{H^s_\omega(\mathbb{R})}.
\end{align}
\end{lemma}
\begin{proof}
There exists a unique $u \in H^s_\omega(\mathbb{R})$ such that, for all $v \in H^s_\omega(\mathbb{R})$  \cite[Ch.~D.3, Th.~2]{Evans2010},
\begin{align}
(u, v)_{H^s_\omega(\mathbb{R})} = \langle \mathcal{L} u, v \rangle_{H^s_\omega(\mathbb{R})^*,H^s_\omega(\mathbb{R})} = \langle u^*, v \rangle_{H^s_\omega(\mathbb{R})^*,H^s_\omega(\mathbb{R})},
\end{align}
and $\| u^*\|_{H^s_\omega(\mathbb{R})^*} = \| u \|_{H^s_\omega(\mathbb{R})}$. Hence, the Riesz map is $u = \mathcal{L}^{-1} u^*$. Thus,
\begin{align}
\| u^*\|_{H^s_\omega(\mathbb{R})^*} = \| u \|_{H^s_\omega(\mathbb{R})} = \|  \mathcal{L}^{-1} u^* \|_{H^s_\omega(\mathbb{R})} \leq \| \mathcal{L}^{-1} \|_{\mathcal{B}(H^s_\omega(\mathbb{R}),H^s_\omega(\mathbb{R}))} \| u^* \|_{H^s_\omega(\mathbb{R})}.
\end{align}
\end{proof}

\begin{theorem}[Implicit Euler convergence]
\label{th:implicit-euler}
Consider the fractional heat equation \cref{eq:heat} and choose the uniform time discretization points $t_0, \dots, t_J$. Define $\delta t =t_1 - t_0$ and let ${\bm \Phi}_N(x)$ denote the truncated quasimatrix of a frame on the Hilbert space $H = H^s_\omega(\mathbb{R})$ with $H^s_\omega(\mathbb{R})$ as defined in \cref{lem:bounded}. Let $\mathcal{L}_{\delta t} = (\mathcal{I} + \delta t (-\Delta)^s)$ which corresponds to an implicit Euler discretization in time. Fix ${\bm \Phi}_N^*(x) = \mathcal{L}_{\delta t}{\bm \Phi}_N(x)$. Suppose that $\| \mathcal{L}^{-1}_{\delta t} \|_{{\mathcal{B}(H^*,H)}}$ and  $\| \mathcal{L}_1^{-1} \|_{{\mathcal{B}(H,H)}}$ are bounded and there exists a Lipschitz constant $L>0$ such that for any $w_1, w_2 \in H$, $\|(-\Delta)^s (w_1 - w_2) \|_{H^*} \leq L \| w_1 - w_2\|_H$. Fix the number of collocation points $M$ and the required SVD tolerance $\epsilon$. Define $h = \max_{j \in \{0,\dots,J\}} h_j$ with $h_j$ being defined as in \cref{eq:timestep}. We denote the exact solution of \cref{eq:heat} by $u(x,t)$. Suppose that we have a ${\mathbf u}_0$ and $0<h_0 \leq h <1$ such that 
\begin{align}
\| u_0 - {\bm \Phi}_N {\mathbf u}_{0}\|_{H} \leq h_0 \|  u_0\|_{H}.
\end{align}
Then, for some $C_0> 0$, the following error bound holds:
\begin{align}
\| u(\cdot, t_j) - {\bm \Phi}_N {\mathbf u}_j \|_{H} \leq   C_0 \delta t  +C_1(u_0,j) C_2(h,j) h,
\end{align}
where 
$
C_2(h,j)=   \sum_{i=0}^{j}  (1+h)^{i}$ and $
C_1(u_0,j) =  \| \mathcal{L}^{-1}_{\delta t}\|^{j}_{\mathcal{B}(H^*,H)}  \| \mathcal{L}^{-1}_1\|^{j}_{\mathcal{B}(H,H)} \| u_0\|_H.
$
\end{theorem}
\begin{proof}
Let $u_j(x)$, $j \in \{0, \dots, J\}$, denote the solutions of
\begin{align}
\frac{u_{j+1}(x) - u_j(x)}{\delta t} + (-\Delta)^s u_{j+1}(x) = 0, \;\; u_0(x) = u(x, t_0).
\end{align}
By assumption $(-\Delta)^s : H \to H^*$ is Lipschitz continuous and thus there exists a constant $C_0>0$ independent of $\delta t$ such that
\begin{align}
\| u(\cdot, t_j) - u_j \|_H \leq C_0 \delta t.
\label{eq:euler4}
\end{align}
\cref{eq:euler4} follows from classical results for implicit Euler discretizations, e.g.~\cite[Sec.~5]{Farago2013} where the Eucledian norms have been replaced by $\| \cdot \|_H$ and $\| \cdot \|_{H^*}$ as appropriate.  Hence,
\begin{align}
\label{eq:euler1}
\| u(\cdot, t_j) - {\bm \Phi}_N {\mathbf u}_j \|_H \leq C_0 \delta t + \| u_j  - {\bm \Phi}_N {\mathbf u}_j \|_H.
\end{align}
Now, by recalling the implicit Euler time step \cref{eq:timestep} and utilizing \cref{lem:bounded},
\begin{align}
\begin{split}
\label{eq:euler2}
&\| u_j  - {\bm \Phi}_N {\mathbf u}_j \|_H = \| \mathcal{L}^{-1}_{\delta t} (\mathcal{L}_{\delta t}u_j  - \mathcal{L}_{\delta t}{\bm \Phi}_N {\mathbf u}_j) \|_H\\
&\indent \leq \| \mathcal{L}^{-1}_{\delta t}\|_{\mathcal{B}(H^*,H)} \| u_{j-1} - {\bm \Phi}^*_N {\mathbf u}_j \|_{H^*}\\
&\indent \leq  \| \mathcal{L}^{-1}_{\delta t}\|_{\mathcal{B}(H^*,H)} \left( \|  {\bm \Phi}_N {\mathbf u}_{j-1}   - {\bm \Phi}^*_N {\mathbf u}^*_{j-1} \|_{H^*} + \| u_{j-1} - {\bm \Phi}_N {\mathbf u}_{j-1} \|_{H^*} \right)\\
&\indent \leq  \| \mathcal{L}^{-1}_{\delta t}\|_{\mathcal{B}(H^*,H)} \left( h \|  {\bm \Phi}_N {\mathbf u}_{j-1} \|_{H^*} + \| \mathcal{L}_1^{-1} \|_{\mathcal{B}(H,H)} \| u_{j-1} - {\bm \Phi}_N {\mathbf u}_{j-1} \|_{H} \right) \\
&\indent \leq  
h \sum_{i=0}^{j-1} \| \mathcal{L}^{-1}_{\delta t}\|^{j-i}_{\mathcal{B}(H^*,H)}  \| \mathcal{L}^{-1}_1\|^{j-i-1}_{\mathcal{B}(H,H)} \|  {\bm \Phi}_N {\mathbf u}_{i} \|_{H^*} \\
&\indent \indent \indent \indent +   \| \mathcal{L}^{-1}_{\delta t}\|^{j}_{\mathcal{B}(H^*,H)}  \| \mathcal{L}^{-1}_1\|^{j}_{\mathcal{B}(H,H)} \| u_{0} - {\bm \Phi}_N {\mathbf u}_{0} \|_H.
\end{split}
\end{align}
We note that, for any $i \in \{0, 1, \dots, j-1\}$,
\begin{align}
\label{eq:euler3}
\begin{split}
& \|  {\bm \Phi}_N {\mathbf u}_{i} \|_{H^*} \leq   \| \mathcal{L}^{-1}_1\|_{\mathcal{B}(H,H)} \|  {\bm \Phi}_N {\mathbf u}_{i} \|_{H} 
=  \| \mathcal{L}^{-1}_1\|_{\mathcal{B}(H,H)} \|  {\bm \Phi}_N {\mathbf u}^*_{i-1} \|_{H} \\
&\indent  \leq  \| \mathcal{L}^{-1}_1\|_{\mathcal{B}(H,H)}  \| \mathcal{L}^{-1}_{\delta t}\|_{\mathcal{B}(H^*,H)}  \|  {\bm \Phi}^*_N {\mathbf u}^*_{i-1} \|_{H^*} \\
& \indent  \leq (1+h)  \| \mathcal{L}^{-1}_1\|_{\mathcal{B}(H,H)}  \| \mathcal{L}^{-1}_{\delta t}\|_{\mathcal{B}(H^*,H)}  \|  {\bm \Phi}_N {\mathbf u}_{i-1} \|_{H^*}\\
& \indent  \leq  (1+h)^i  \| \mathcal{L}^{-1}_1\|_{\mathcal{B}(H,H)}^i  \| \mathcal{L}^{-1}_{\delta t}\|^i_{\mathcal{B}(H,H^*)}  \|  {\bm \Phi}_N {\mathbf u}_{0} \|_{H^*} \\
& \indent  \leq    (1+h)^{i+1}  \| \mathcal{L}^{-1}_1\|_{\mathcal{B}(H,H)}^{i+1}  \| \mathcal{L}^{-1}_{\delta t}\|^i_{\mathcal{B}(H,H^*)} \| u_0 \|_{H},
\end{split}
\end{align}
By assumption $ \| u_{0} - {\bm \Phi}_N {\mathbf u}_{0} \|_H \leq h \| u_0 \|_H$. Hence, by utilizing the bound in \cref{eq:euler3} for \cref{eq:euler2} and the subsequent bound of $\| u_j  - {\bm \Phi}_N {\mathbf u}_j \|_H$ into \cref{eq:euler1}, we deduce the result. 
\end{proof}

\section{Examples}
\label{sec:examples}
In this section we provide several numerical examples. In all the examples we choose collocation points $\{x_j\}_{j=1}^M$ and fix the linear functionals in \cref{alg:frame} as $\ell_j(f) = f(x_j)$. 

\textbf{Code availability:} For reproducibility, an implementation of the solver as well as scripts to generate the plots and solutions can be found in FractionalFrames.jl \cite{fractionalframes.jl} and archived on Zenodo \cite{fractionalframes-zenodo-v5}. The implementation is written in Julia \cite{Bezanson2017} and heavily relies on the ClassicalOrthogonalPolynomials.jl \cite{ClassicPoly.jl2023} and HypergeometricFunctions.jl \cite{Hyper.jl2023,slevinsky2025} packages.

\subsection{The Gaussian}
\label{sec:gaussian}
In our first example we consider the exact solution $u(x) =  \E^{-x^2}$. We note that \cite[Prop.~4.2]{Sheng2020},
\begin{align}
(\mathcal{I} + (-\Delta)^{s}) u(x)  =  \E^{-x^2} + 4^s \frac{\Gamma(s+1/2)}{\Gamma(1/2)} {_1}F_1(s+1/2;1/2;-x^2), \label{eq:ex2-1}
\end{align}
where ${_1}F_1$ denotes the Kummer confluent hypergeometric function \cite[Sec.~16.2]{OlverNIST}. We investigate \cref{eq:ex2-1} when $s=1/3$.

We pick $\bigcup_{k=1}^5 \{ \eP_n^{I_k,(-1/3,-1/3)} \}_{n=0}^\infty \cup \{ Q_n^{I_k, (1/3,1/3)} \}_{n=0}^\infty$, as the solution family of functions, where $I_1,\dots,I_5$ are $[-5,-3]$, $[-3,-1]$, $[-1,1]$, $[1,3]$, and $[3,5]$, respectively. $\eP^{I_k}_n(x)$ and $Q^{I_k}_n(x)$ denote the affine-transformed extended Jacobi functions and weighted Jacobi polynomials, respectively, as defined in \cref{sec:jacobi} and \cref{def:affine}. Let $N$ denote the number of functions in the approximation frame. We choose $N$ equally spaced points in $[a+\epsilon, b-\epsilon]$, $\epsilon = 10^{-2}$, where $a, b$ represent the endpoints of each interval $I_k$ as well as $N$ equally spaced points in $[-10+\epsilon,-5-\epsilon]$ and $[5+\epsilon,10-\epsilon]$. This results in $7N$ collocation points.  We have chosen the intervals $I_k$ and the endpoints $\{-10,10\}$ to completely cover the support of the right-hand side prior to where the asymptotic decay dominates. For this example, we found that equally spaced points produced similar results when compared to Chebsyshev-like distributions. Hence, we opt for equally spaced points within the intervals for this example and the subsequent one-dimensional examples in this section. 

A semi-log plot of the convergence for the right-hand side and the solution is given in \cref{fig:gaussian:convergence}. The error is measured in the $\ell^\infty$-norm as measured at the collocation points. We observe approximately spectral convergence as we include more functions in the expansion, with the right-hand side error  plateauing when it reaches an order of $10^{-14}$ in magnitude and the solution error at $10^{-11}$. 

Since we are not guaranteed that the family of functions is a frame on all $\mathbb{R}$, in \cref{fig:gaussian:coeffs} we investigate the behaviour of the expansion coefficients of the right-hand side. In particular we plot the $\ell^\infty$-norm of the coefficient vector as we increase the number of functions in the expansion. Despite the lack of a strict frame condition, we achieve bounded coefficients with a magnitude on the order of one for $N>40$.   We observe that the error plateau from $N=150$ coincides with the ``settling down" of the norm of the coefficient vector. We conjecture that for $N>150$, the least squares solver is discarding singular values associated with the higher order (weighted) Jacobi functions. It may be possible to develop a stopping criterion based on this observation although this lies outside the scope of this work.   
\begin{figure}[h!]
\centering
\subfloat[$\ell^\infty$-norm error.]{\includegraphics[width =0.43 \textwidth]{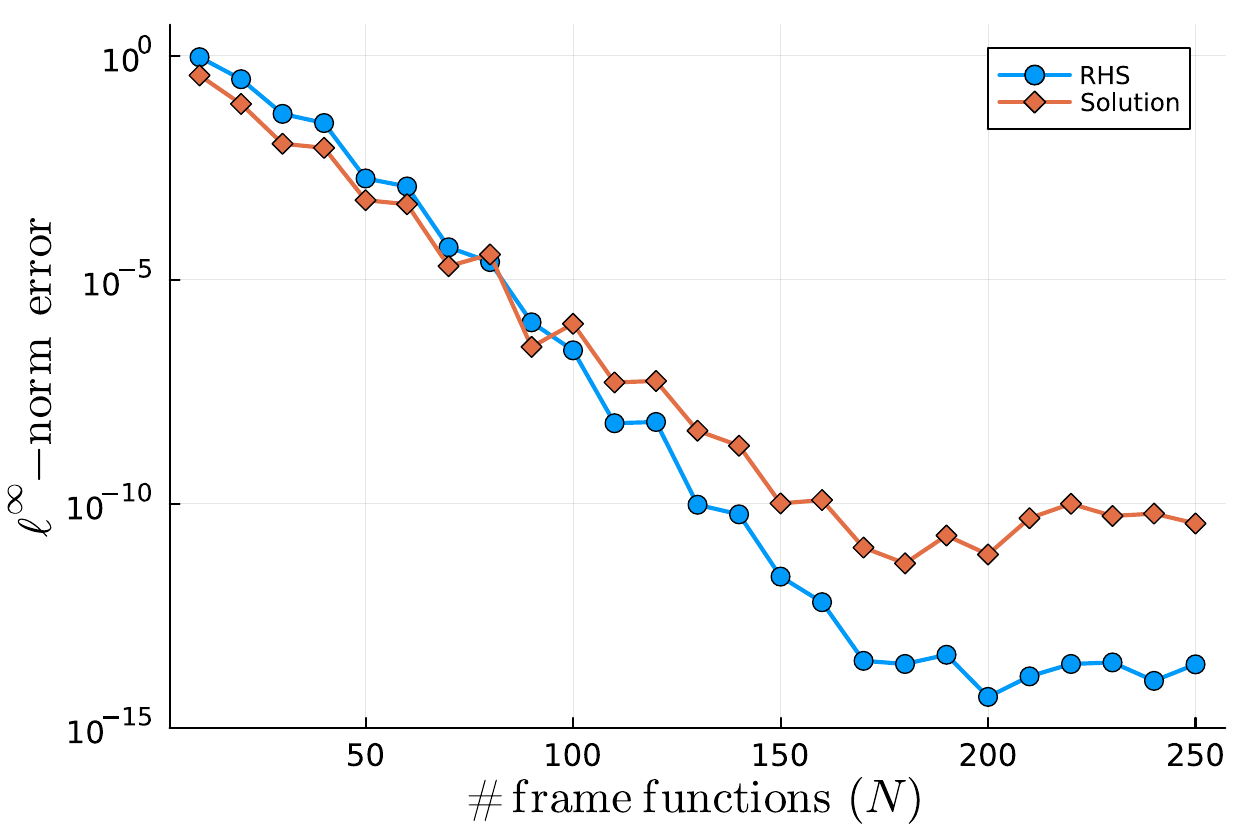} \label{fig:gaussian:convergence}}
\subfloat[$\ell^\infty$-norm of ${\mathbf u}$.]{\includegraphics[width =0.43 \textwidth]{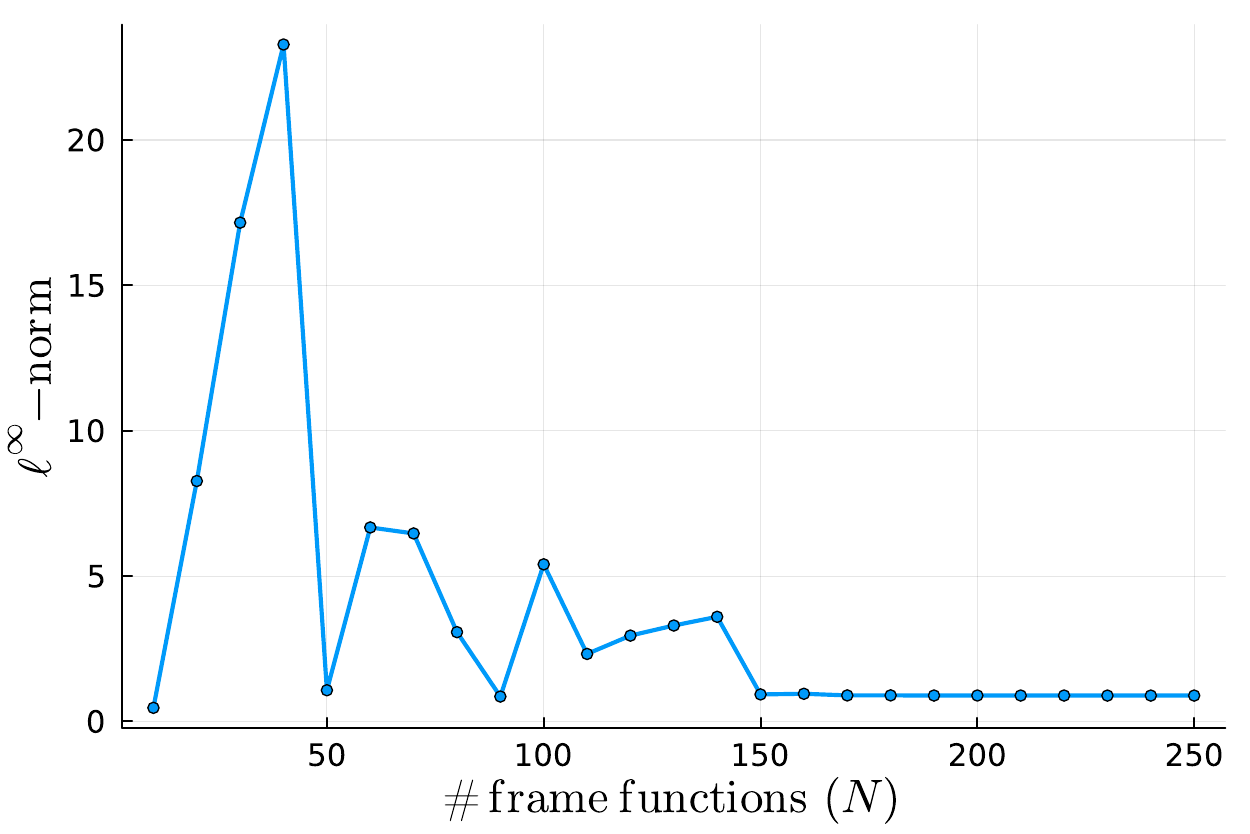} \label{fig:gaussian:coeffs}}
\caption{(Left) $\ell^\infty$-norm error in the right-hand side and solution of the Jacobi approximation for \cref{eq:ex2-1} with $s=1/3$. We observe approximate spectral convergence with the error stagnating when it reaches an order of $10^{-14}$ in magnitude for the right-hand side and $10^{-11}$ for the solution. (Right) $\ell^\infty$-norm of the coefficient vector of the expansion. We observe that coefficient norm is small.}\label{fig:gaussian}
\end{figure}

\subsection{Multiple exponents}
\label{sec:multiple-exponents}

In this section we consider an equation with fractional Laplacians with different exponents. Consider the equation:
\begin{align}
(\mathcal{I} + (-\Delta)^{1/3} + (-\Delta)^{1/5}) u(x) = \E^{-x^2} + (-\Delta)^{1/3}\E^{-x^2} + (-\Delta)^{1/5}\E^{-x^2}.
\label{eq:mult-exps}
\end{align}
As in the previous example, the exact solution to this equation is $u(x) =  \E^{-x^2}$. The sum space we consider is $\bigcup_{k=1}^5 \{ \eP_j^{I_k,(-1/4,-1/4)} \}_{j=0}^\infty \cup \{ Q_j^{I_k, (1/4,1/4)} \}_{j=0}^\infty$ where the intervals $I_1,\dots,I_5$ and the choice of collocation points are the same as in \cref{sec:gaussian}.  One could instead consider a sum space consisting of extended Jacobi functions with both the weights $(1/3,1/3)$ and $(1/5,1/5)$, and correspondingly their negative counterparts for the weighted Jacobi polynomials, in order to match the fractional exponents. However, our goal is to exemplify that the method is robust even if the weights do not match the fractional exponents. The weight chosen in this section's sum space is chosen as an intermediate between $1/3$ and $1/5$. 

A semi-log plot of the convergence for the right-hand side and the solution is given in \cref{fig:mult-exp:convergence}. The error is measured in the $\ell^\infty$-norm as measured at the collocation points. We observe approximately spectral convergence until $N=150$ where the convergence rate stagnates. The best right-hand side error is on the order of $10^{-13}$ in magnitude and the best solution error is on the order of $10^{-10}$. We observe that as $N$ increases the error becomes worse in the solution. As in the previous example we also plot the $\ell^\infty$-norm of the coefficient vector in \cref{fig:mult-exp:coeffs} and observe that the magnitude is small, reaching an order of one when $N>50$. We expect that fine-tuning the weight parameters in the Jacobi polynomials would allow one to achieve faster convergence.

\begin{figure}[h!]
\centering
\subfloat[$\ell^\infty$-norm error.]{\includegraphics[width =0.42 \textwidth]{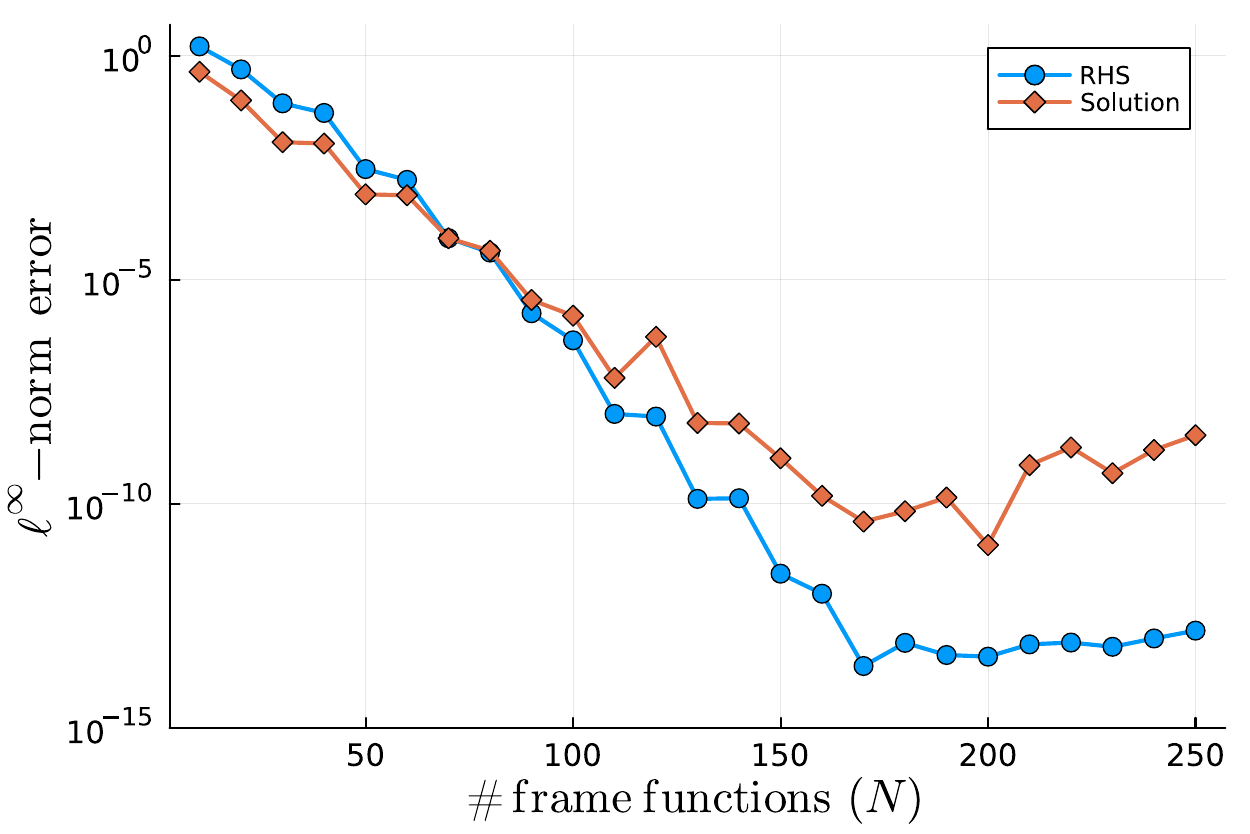} \label{fig:mult-exp:convergence}}
\subfloat[$\ell^\infty$-norm of the coefficient vector.]{\includegraphics[width =0.42 \textwidth]{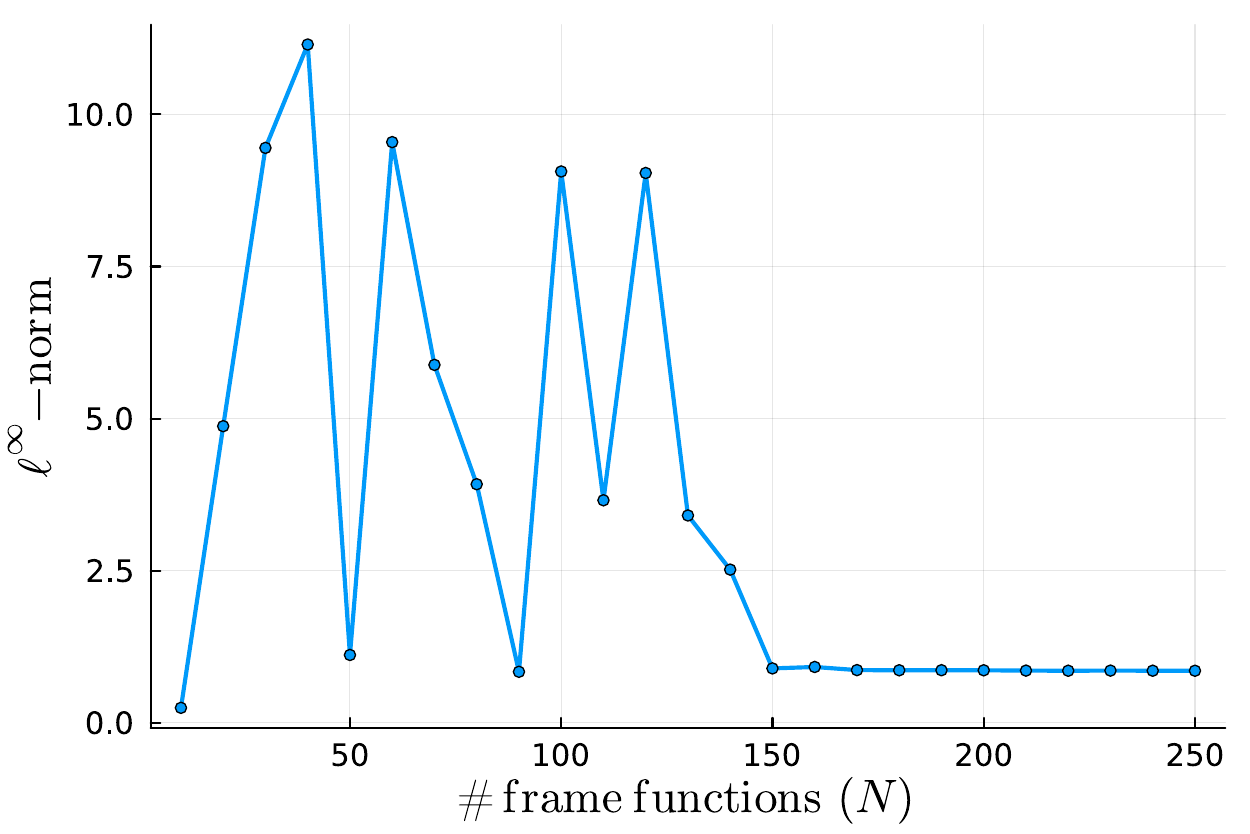} \label{fig:mult-exp:coeffs}}
\caption{(Left) Pointwise error in the right-hand side and solution of the Jacobi approximation for \cref{eq:mult-exps}. We observe approximate spectral convergence with the error stagnating when it reaches an order of $10^{-13}$ in magnitude for the right-hand side and $10^{-10}$ for the solution. (Right) $\ell^\infty$-norm of the coefficient vector of the expansion. We observe that coefficient norm is small.}\label{fig:mult-exp}
\end{figure}

\subsection{Piecewise smooth right-hand side}
\label{sec:piecewise-smooth}

Here we explore the convergence of our solver when the right-hand side is no longer smooth. We consider the problem
\begin{align}
(\mathcal{I} + (-\Delta)^{1/5}) u(x) = 
\begin{cases}
\frac{1}{1+x^2} & \text{if} \; x< 0,\\
\E^{-x} & \text{if} \; x \geq 0.
\end{cases}
\label{eq:piecewise-smooth}
\end{align}
The right-hand side is smooth wherever $x \neq 0$ and continuous, but not continuously differentiable, at $x=0$. 

We pick $\bigcup_{k=1}^5 \{ \eP_n^{I_k,(-1/5,-1/5)} \}_{n=0}^\infty \cup \{ Q_n^{I_k, (1/5,1/5)} \}_{n=0}^\infty$, as the solution family of functions, where $I_1,\dots,I_8$ are $[-16,-12]$, $[-12,-8]$, $[-8,-4]$, $[-4,0]$, $[0,4]$, $[4,8]$, $[8,12]$, and $[12,16]$, respectively. Akin to before we choose $N$ equally spaced points in $[a+\epsilon, b-\epsilon]$, $\epsilon = 10^{-2}$, where $a, b$ represent the endpoints of each interval $I_k$ as well as $N$ equally spaced points in $[-25+\epsilon,-16-\epsilon]$ and $[16+\epsilon,25-\epsilon]$. This results in $10N$ collocation points.

A semi-log plot of the convergence for the right-hand side and the solution is given in \cref{fig:piecewise-smooth:convergence} where, as the exact solution is not known, we measure the Cauchy error between each successive approximate solution. The Cauchy error is measured in the $\ell^\infty$-norm at the collocation points of the finer solution. We observe approximately spectral convergence in the approximation of the right-hand side as measured via the Cauchy error and against the exact function whereas the Cauchy error of the solution follows an algebraic decay. The right-hand side error stagnates when it reaches an order of $10^{-9}$ in magnitude and the solution error at $10^{-7}$.

In \cref{fig:piecewise-smooth:coeffs} we investigate the behaviour of the expansion coefficients of the right-hand side which reaches its peak in magnitude at $N=180$ and then decreases.
\begin{figure}[h!]
\centering
\subfloat[(Cauchy) $\ell^\infty$-norm error.]{\includegraphics[width =0.43 \textwidth]{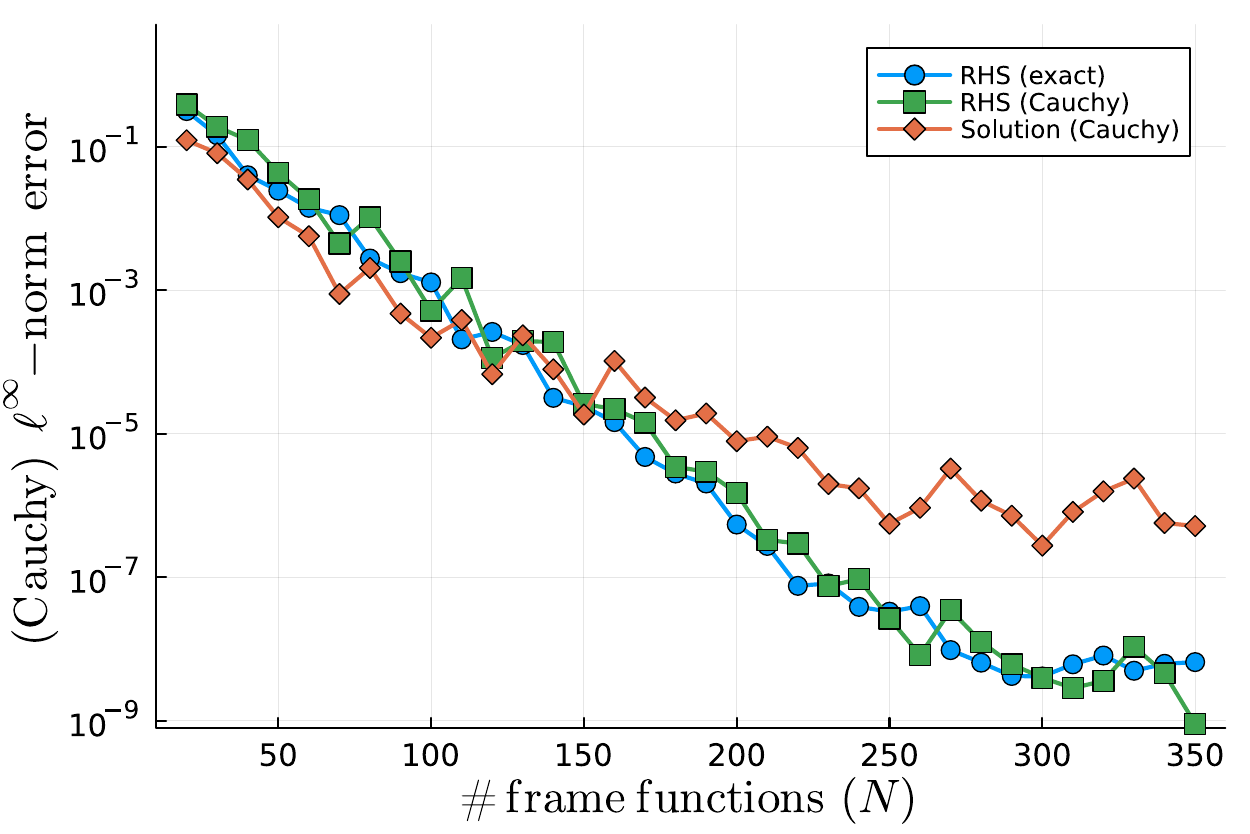} \label{fig:piecewise-smooth:convergence}}
\subfloat[$\ell^\infty$-norm of ${\mathbf u}$.]{\includegraphics[width =0.43 \textwidth]{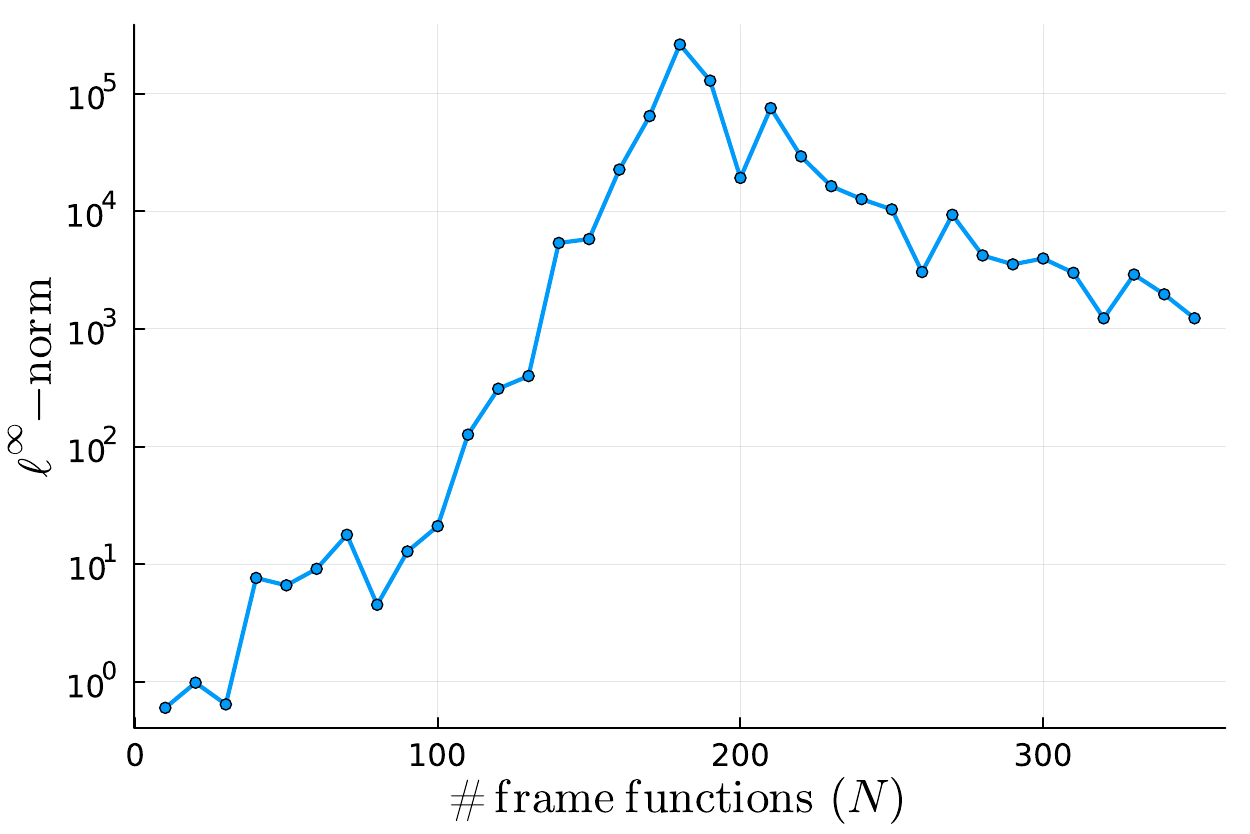} \label{fig:piecewise-smooth:coeffs}}
\caption{(Left) $\ell^\infty$-norm error in the right-hand side as measured against the exact solution at the collocation points as well as the $\ell^\infty$-norm error in the right-hand side and solution against the previous expansion (Cauchy error) for \cref{eq:piecewise-smooth}. We observe convergence with the error stagnating when it reaches an order of $10^{-9}$ in magnitude for the right-hand side and $10^{-7}$ for the solution. (Right) $\ell^\infty$-norm of the coefficient vector of the expansion.}\label{fig:piecewise-smooth}
\end{figure}

\subsection{Two-dimensional Gaussian}
\label{sec:2d-gaussian}

In this example we exemplify that the spectral method extends to two-dimensional problems. Consider the equation
\begin{align}
(-\Delta)^{1/2} u(x,y) = 2 \Gamma(3/2) {_1}F_1(3/2;1;-x^2-y^2), \label{eq:2dgaussian}
\end{align}
which has the exact solution $u(x,y) = \exp(-x^2-y^2)$ \cite[Prop.~4.2]{Sheng2020}. We utilize a solution family of functions consisting of weighted and extended Zernike polynomials and functions (as introduced in \cref{sec:zernike}), respectively, that are scaled radially on a sequence of nested disks. Let $W^{(b)}_{n,m,j}(x,y) \coloneqq (1-r^2)^b Z^{(b)}_{n,m,j}(x,y)$. Since the right-hand side and the solution are rotationally invariant, both functions are best approximated by the family of functions restricted to the $(m,j) = (0,1)$ Fourier mode and sign. Thus we use the sum space
\begin{align}
\bigcup_{k=1}^K \{ W^{(1/2)}_{n,0,1}(a_k x, a_k y) \} \cup  \{ \tilde{Z}^{(-1/2,-1/2)}_{n,0,1}(a_k x, a_k y) \}, \;\; a_k  \in \{1, 3/2, 2, 3, 4\}. 
\end{align}

The supports of the weighted Zernike polynomials are overlapping unlike for the weighted Jacobi polynomials in the one-dimensional case. In order to leverage non-overlapping domains, one would require a domain decomposition consisting of an innermost disk and concentric annuli together with corresponding families of functions defined on such domains. For instance one may use Zernike polynomials for the innermost disk followed by Zernike annular polynomials for the annuli \cite{papadopoulos2024}. However, we currently do not have the explicit formulae for the fractional Laplacian applied to Zernike annular polynomials. Although these could be evaluated numerically, this may cause a computational bottleneck.

As the sum space contains functions that are undefined for $r \in \{0, 1, 3/2, 2, 3, 4\}$, we choose $N$ equally spaced points in the radial direction $[a+\epsilon, b-\epsilon]$, $\epsilon = 10^{-3}$, where $a, b$ represent the endpoints of each interval as well as $N$ equally spaced points in $[4+\epsilon,10-\epsilon]$. We then cross product these radial collocation points with 5 equally spaced angular collocation points between 0 to $2\pi$. This results in $30N$  collocation points. A semi-log plot of the convergence for the solution is given in \cref{fig:2d-gaussian:convergence}. The error is measured in the $\ell^\infty$-norm as measured at the collocation points. The best solution error is on the order of $10^{-9}$ in magnitude. We plot the $\ell^\infty$-norm of the coefficient vector in \cref{fig:2d-gaussian:coeffs} and observe an order of one in magnitude when $N \leq 100$. 

\begin{figure}[h!]
\centering
\subfloat[$\ell^\infty$-norm error.]{\includegraphics[width =0.4 \textwidth]{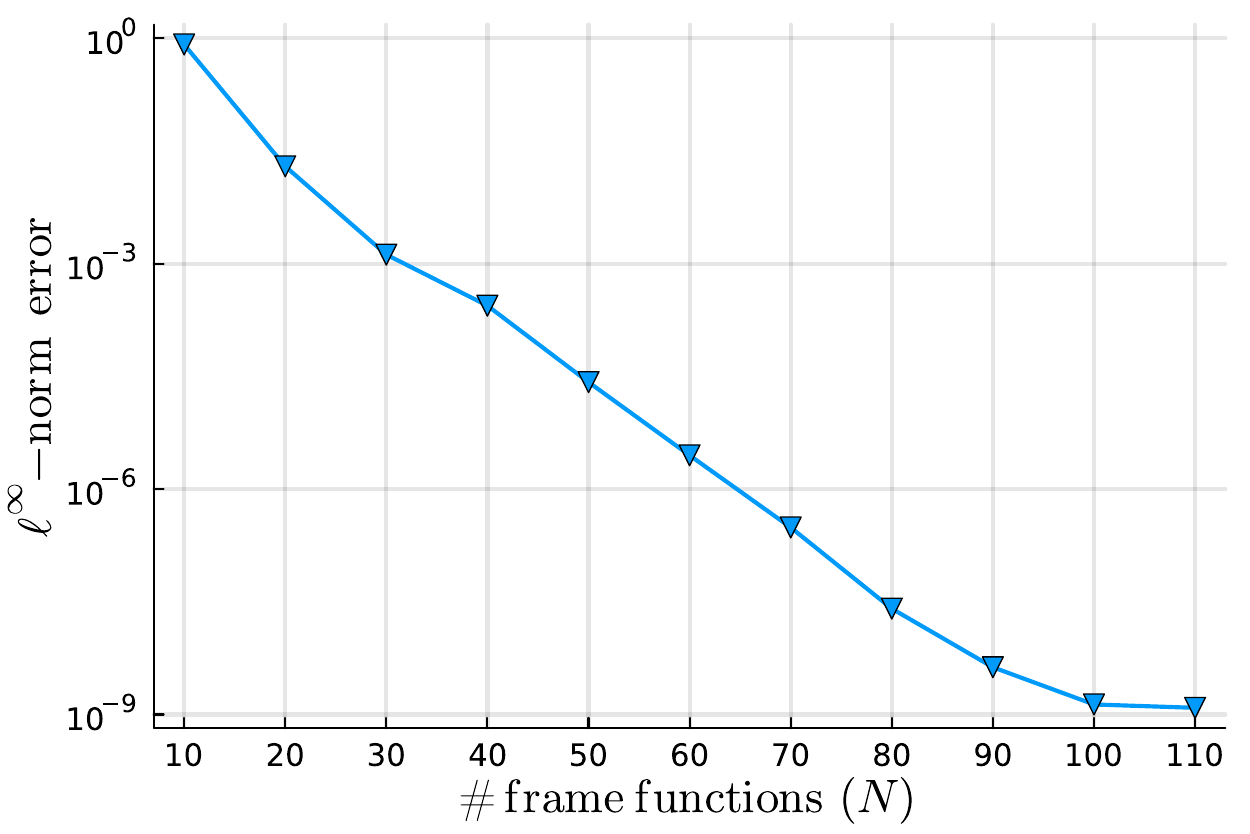} \label{fig:2d-gaussian:convergence}}
\subfloat[$\ell^\infty$-norm of the coefficient vector.]{\includegraphics[width =0.4 \textwidth]{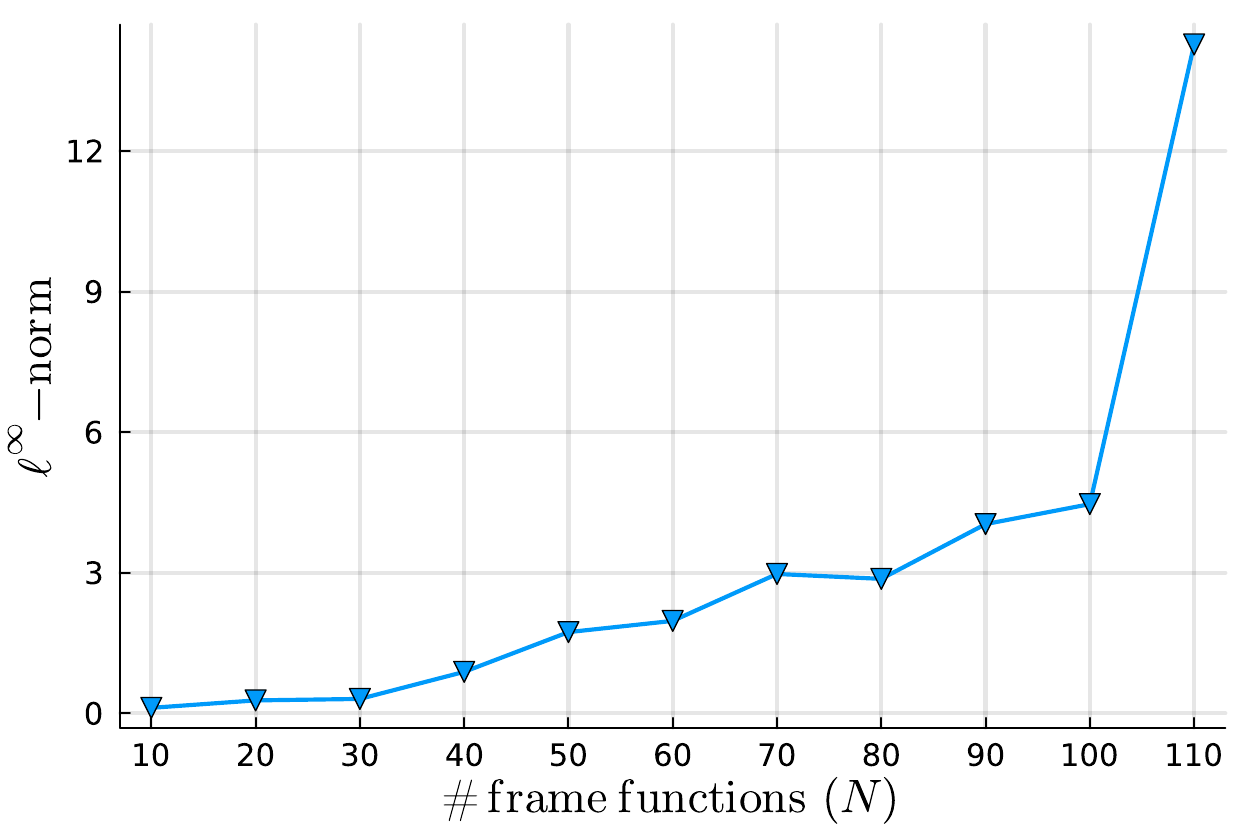} \label{fig:2d-gaussian:coeffs}}
\caption{(Left) Pointwise error in the solution of the Zernike approximation for \cref{eq:2dgaussian}. We observe approximate spectral convergence with the error stagnating when it reaches an order of $10^{-9}$ in magnitude. (Right) $\ell^\infty$-norm of the coefficient vector of the expansion. We observe that coefficient norm is small.}\label{fig:2d-gaussian}
\end{figure}
\subsection{Fractional heat equation}
\label{sec:fractional-heat}
We consider the time-dependent fractional heat equation \cref{eq:heat} with the initial state $u(x,0) = (1+x^2)^{-1}$ and $s=1/2$. The explicit solution is $u(x,t) = (1 + t) / (x^2 + (1+t)^2)$. We consider the time discretizations backward Euler (1$^{\text{st}}$ order), implicit midpoint rule (2$^{\text{nd}}$ order), Gauss--Legendre (4$^{\text{th}}$ order), and Gauss--Legendre (6$^{\text{th}}$ order) via \cref{alg:rk}.

Our goal is to examine the convergence of the temporal discretization. We utilize the solution family of functions  $\bigcup_{k=1}^5 \{ \eP_j^{I_k,(-1/2,-1/2)} \}_{j=1}^\infty \cup \{ Q_j^{I_k,(1/2,1/2)} \}_{j=0}^\infty$ where the intervals $I_1,\dots,I_5$ are $[-5,-3]$, $[-3,-1]$, $[-1,1]$, $[1,3]$, and $[3,5]$, respectively. We choose 2000 equally spaced points in $[a+\epsilon, b-\epsilon]$, $\epsilon = 10^{-4}$, where $a, b$ represent the endpoints of each interval as well as 2000 equally spaced points in $[-20+\epsilon,-5-\epsilon]$ and $[5+\epsilon,20-\epsilon]$. This results in 14,000 collocation points. We include 250 functions in the family of functions.  Expanding the initial state results in an $\ell^\infty$-norm error of $3.16 \times 10^{-12}$ as measured at the collocation points. Next we measure the relative maximum $\ell^\infty$-norm as measured at all the collocation points and time steps from $t=0$ to $t=1$ for time step choices $\delta t \in  \{10^{-1}, 10^{-2}, 10^{-3},  10^{-4}\}$, i.e.
\begin{align}
\max_{k \in \{0, 1, 2, \dots, \delta t^{-1} \}} \max_{x \in X_c} \frac{| u(x,k \delta t) - {\bm \Phi}(x) \vectt{v}_{k} |}{|u(x,k \delta t)|},
\label{eq:time-error}
\end{align}
where $X_c$ is the set of the collocation points. The convergence results are given in the left panel in \cref{fig:fractional-heat-convergence}. We observe the expected order of convergence for the methods, with deviations occurring when the error is close to a magnitude of the order of $10^{-10}$ where the error in the spatial discretization likely causes the degradation in the convergence. This example is numerical evidence for the result provided in \cref{th:implicit-euler}. Moreover, it appears that the result extends to all implicit Runge--Kutta methods and their classical associated convergence rates. In the right panel of \cref{fig:fractional-heat-convergence} we plot the pointwise error of the solution, computed via the 6$^{\text{th}}$ order Gauss--Legendre method with $\delta t = 10^{-2}$, at $t=1$ in the domain $x \in [-10^3,10^3]$. We observe that the algebraic tails of the solution are accurately captured by this spectral method.
 
\begin{figure}[h!]
\centering
\includegraphics[width =0.57 \textwidth]{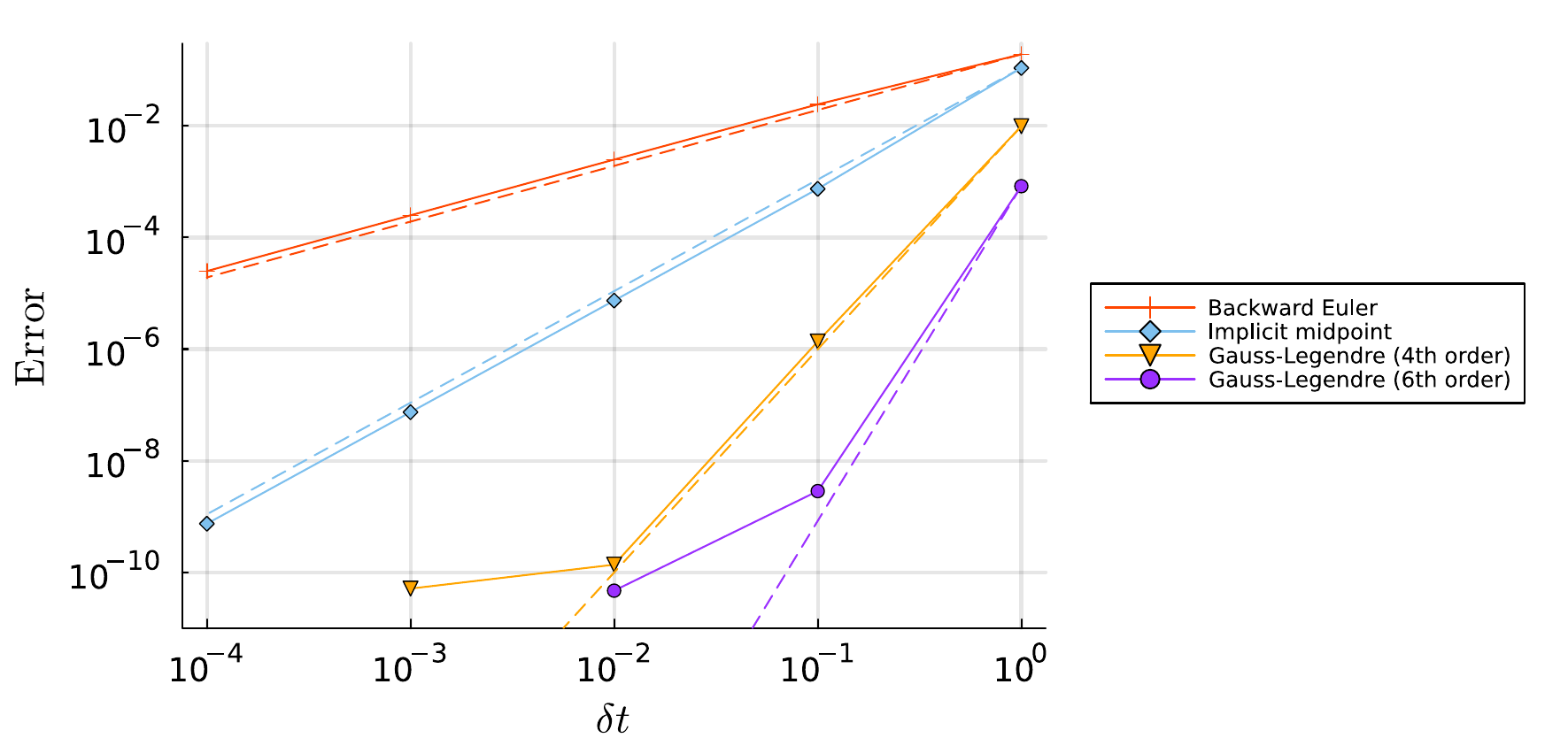}
\includegraphics[width =0.42 \textwidth]{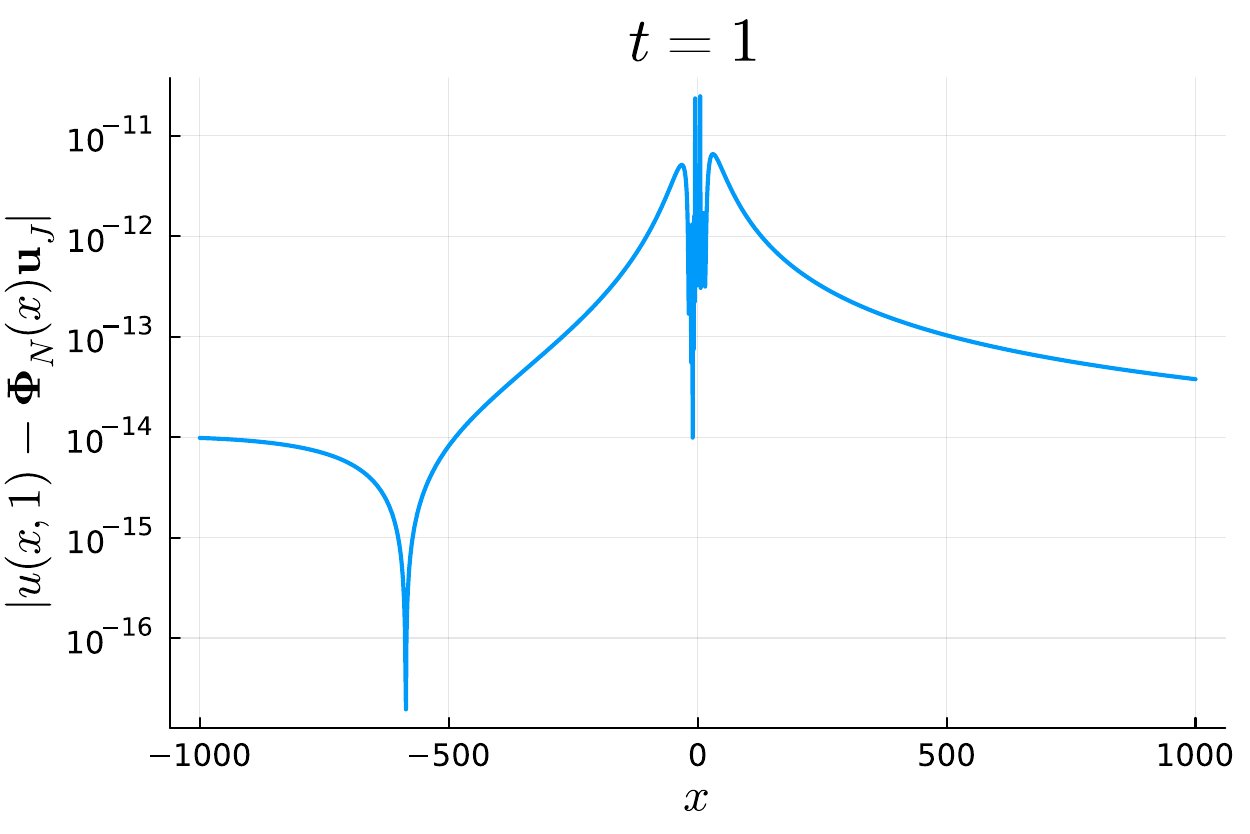} 
\caption{(Left) Convergence of the four Runge--Kutta methods for decreasing $\delta t$ using \cref{eq:time-error} to measure the error. The red, blue, orange, and magenta dashed lines indicate $\mathcal{O}(\delta t)$, $\mathcal{O}(\delta t^2)$, $\mathcal{O}(\delta t^4)$, and $\mathcal{O}(\delta t^6)$ convergence rates, respectively. We observe the expected convergence rates of all four methods. We note that the stagnation of the convergence rate when close to an order of magnitude of $10^{-10}$ is likely caused by the error in the spatial discretization. (Right) Pointwise error of the approximated solution, computed via the $6^{\text{th}}$-order Gauss--Legendre method with $\delta t = 10^{-2}$, at $t=1$ in the domain $x \in [-10^3,10^3]$. We see that the algebraic tails are suitably captured.}
\label{fig:fractional-heat-convergence}
\end{figure}
\subsection{Variable fractional exponent}
\label{sec:variable-s}

Consider the heat equation \cref{eq:heat}, except now the exponent of the fractional Laplacian also evolves in time \cite{Patnaik2020}, i.e.
\begin{align}
\partial_t u(x,t) + (-\Delta)^{s(t)} u(x,t) = 0, \;\; u(x,0) = (1+x^2)^{-1}.
\label{eq:variable-s}
\end{align}
In this example we pick $s(t) = s^*(t) \coloneqq 1/2 - t/3$. We choose the same intervals and collocation points as in \cref{sec:gaussian} with $N= 250$ resulting in 1750 collocation points. Utilizing an implicit midpoint rule for the temporal discretization with $\delta t = 10^{-2}$, we consider the family of functions $\bigcup_{k=1}^5 \{ \eP^{I_k, (-s(t_n), -s(t_n))}_n \}_{n=1}^\infty \cup \{ Q^{I_k, (s(t_n), s(t_n))}_n \}_{n=0}^\infty$, $t_j = j \delta t$, $j = 0,1,\dots,100$, at each time step. In \cref{fig:variable-s} we plot the approximation of the solution of \cref{eq:variable-s} at $t=1$ together with the equivalent solutions if $s(t) = 1/2$, $1/3$, and $1/6$ for all $t$. We see that the decay of the variable exponent solution is similar to the decay where $s = 1/3$ which is equal to $s^*(1/2)$.

\begin{figure}[h!]
\centering
\includegraphics[width =0.4 \textwidth]{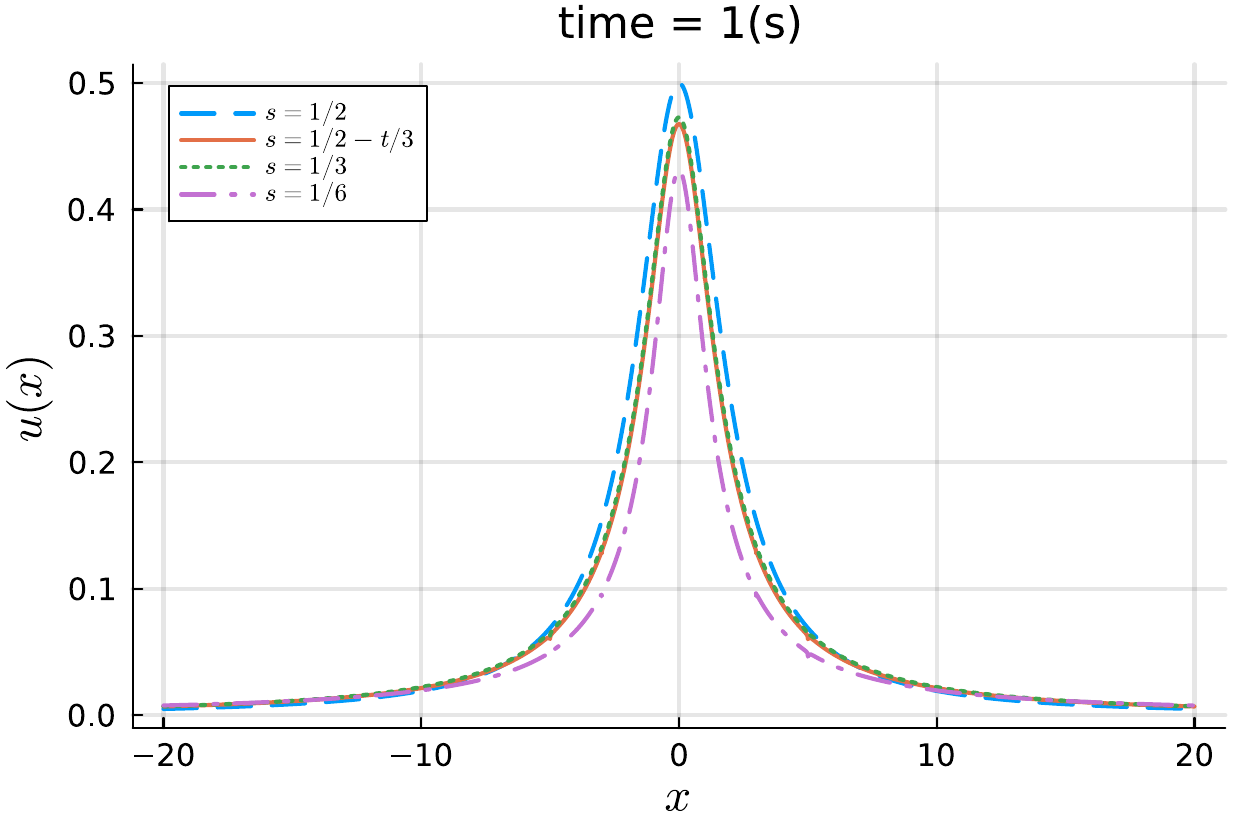}
\caption{The approximation solutions of \cref{eq:variable-s} where $s(t) = 1/2$, $1/2 - t/3$, $1/3$, and $1/6$ at $t=1$.}
\label{fig:variable-s}
\end{figure}

\section{Conclusions}
\label{sec:conclusions}
In this work, we introduced a frame approach, summarized in \cref{alg:frame}, for solving equations involving nonlocal and fractional terms such as the fractional Laplacian $(-\Delta)^s$, $s \in  (  0,1)$.  One fixes a family of functions for the expansion of the solution and expands the right-hand side of the equation in the space induced by applying the equation operator to this family of functions. This  diagonalizes the equation and  reduces the problem from a solve of a nonlocal problem to an interpolation problem.

Under suitable conditions, we showed that if the family of functions for the solution expansion is a frame in a suitable Hilbert space, then the induced family of functions for the right-hand side expansion is a frame in the dual of the Hilbert space. By considering solution expansions in weighted Jacobi and Zernike polynomials, we utilized recent results to deduce the family of functions for the right-hand side expansion. With this knowledge, one automatically recovers favourable properties of frames, in particular, the utilization of a truncated SVD projection to find a well-posed expansion. This culminated in deriving an a priori estimate in \cref{th:apriori}.

This spatial discretization was coupled with Runga--Kutta methods in time in order to develop a method for time-dependent problems including the fractional heat equation as well as a problem where the exponent of the fractional Laplacian evolved in time. We proved that one recovers the optimal convergence rates under an implicit Euler discretization and observed the expected convergence rates for our chosen Runge-Kutta methods in practice, up to a $6^{\text{th}}$ order method. Moreover, the method was also competitive when applied to a two-dimensional problem.

This approach is general and is not limited to the operators considered in this work. Provided one has explicit expressions of the equation operator applied to a family of functions, one may use this frame approach for a considerable number of different problems. Future work will consider generalizing this technique to be competitive in higher dimensions. The bottleneck inevitably becomes the SVD factorization to interpolate the right-hand side. However, by using randomized SVD solvers \cite{Halko2011, Meier2023} and the $AZ$-algorithm \cite{Coppe2020}, we hope to considerably reduce the computational complexity. 

\section*{Funding}  Our research was supported by the EPSRC grant EP/T022132/1 ``Spectral element methods for fractional differential equations, with applications in applied analysis and medical imaging". IPAP and SO were also supported by the Leverhulme Trust Research Project Grant RPG-2019-144 ``Constructive approximation theory on and inside algebraic curves and surfaces". IPAP was further supported by the Deutsche Forschungsgemeinschaft (DFG, German Research Foundation) under Germany's Excellence Strategy -- The Berlin Mathematics Research Center MATH+ (EXC-2046/1, project ID: 390685689). TSG was also supported by a PIMS-Simons postdoctoral fellowship, jointly funded by the Pacific Institute for the Mathematical Sciences and the Simons Foundation. JAC was supported by the Advanced Grant Nonlocal-CPD (Nonlocal PDEs for Complex Particle Dynamics: Phase Transitions, Patterns and Synchronization) of the European Research Council Executive Agency (ERC) under the European Union’s Horizon 2020 research and innovation programme (grant agreement No.~883363).

\printbibliography
\end{document}